\documentclass[11pt]{amsart}
\usepackage{amsmath}
\usepackage{amssymb}
\usepackage{harpoon}
\usepackage{graphicx}
\usepackage{esint}

\theoremstyle{plain}
\newtheorem{theorem}{Theorem}[section]
\newtheorem{lemma}[theorem]{Lemma}

\newtheorem{proposition}[theorem]{Proposition}

\theoremstyle{definition}
\newtheorem{remark}[theorem]{Remark}

\numberwithin{equation}{section}

\theoremstyle{plain}

\numberwithin{equation}{section}

\allowdisplaybreaks

\begin{document}

\title[Geometric Inequalities for Quasi-Local Masses]
{Geometric Inequalities for Quasi-Local Masses}

\author[Alaee]{Aghil Alaee}
\address{Center of Mathematical Sciences and Applications\\
Harvard University\\
20 Garden Street\\
Cambridge, MA 02138, USA}
\email{aghil.alaee@cmsa.fas.harvard.edu}

\author[Khuri]{Marcus Khuri}
\address{Department of Mathematics\\
Stony Brook University\\
Stony Brook, NY 11794, USA}
\email{khuri@math.sunysb.edu}

\author[Yau]{Shing-Tung Yau}
\address{Department of Mathematics\\
Harvard University\\
1 Oxford Street\\
Cambridge, MA 02138, USA}
\email{yau@math.harvard.edu}

\thanks{A. Alaee acknowledges the support of NSERC Postdoctoral Fellowship 502873, the Gordon and Betty Moore Foundation, and the John Templeton Foundation. M. Khuri acknowledges the support of NSF Grant DMS-1708798. S.-T. Yau acknowledges the support of NSF Grant DMS-1607871.}

\begin{abstract}
In this paper lower bounds are obtained for quasi-local masses in terms of charge, angular momentum, and horizon area. In particular we treat three quasi-local masses based on a Hamiltonian approach, namely the Brown-York, Liu-Yau, and Wang-Yau masses. The geometric inequalities are motivated by analogous results for the ADM mass. They may be interpreted as localized versions of these inequalities, and are also closely tied to the conjectured Bekenstein bounds for entropy of macroscopic bodies. In addition, we give a
new proof of the positivity property for the Wang-Yau mass which is used to remove the spin condition in higher dimensions. Furthermore, we
generalize a recent result of Lu and Miao to obtain a localized version of the Penrose inequality for the static Wang-Yau mass.
\end{abstract}

\maketitle

\tableofcontents

\section{Introduction}
\label{sec1} \setcounter{equation}{0}
\setcounter{section}{1}

Based on heuristic arguments Bekenstein \cite{Bekenstein} proposed a universal upper bound on the entropy $\mathcal{S}$ of macroscopic bodies in terms of the radius $\mathcal{R}$ of the smallest sphere enclosing the object and its total energy $\mathcal{E}$. Generalizations including contributions from the angular momentum $\mathcal{J}$ and charge $Q$ of the body were later given in \cite{BekensteinMayo,Hod,Hod1,Zaslavskii} yielding the inequality
\begin{equation}\label{a}
 \sqrt{(\mathcal{E}\mathcal{R})^2 -c^2 \mathcal{J}^2} - \frac{Q^2}{2}
\geq \frac{\hbar c}{2 \pi k_b}\mathcal{S},
\end{equation}
where $c$ is the speed of light, $\hbar$ is the reduced Planck's constant, and  $k_b$ denotes Boltzmann's constant. This inequality is difficult to establish rigorously, and so it is natural to study a simpler estimate that is implied by \eqref{a} which can then serve as a test of validity for the original.
Such a reduced inequality may be obtained by using the fact that entropy is always nonnegative, namely
\begin{equation}\label{Bek}
\mathcal{E}^2\geq \frac{Q^4}{4\mathcal{R}^2}+\frac{c^2\mathcal{J}^2}{\mathcal{R}^2},
\end{equation}
Initial investigations of \eqref{Bek} were made by Dain \cite{Dain}, where the role of $\mathcal{E}$ is played by the ADM mass of an asymptotically flat spacetime containing the relativistic body; further work in this direction may be found in \cite{AngladaGabach-ClementOrtiz,JaraczKhuri}. In the present work we will obtain versions of \eqref{Bek} in which $\mathcal{E}$ is more appropriately represented by various quasi-local masses, specifically those quasi-local masses arising from a Hamiltonian approach including the Brown-York \cite{BrownYork}, Liu-Yau \cite{LiuYau0}, and Wang-Yau \cite{WangYauPMT1} masses. Moreover, the `bodies' that we consider will be allowed to contain black holes.

Another motivation of the present article is to prove localized versions of the Penrose inequality, as well as other well-known geometric inequalities implied by the Penrose inequality e.g. the mass-angular momentum-charge inequality, in which the role of ADM mass is played by the Hamiltonian based quasi-local masses.
It should be observed that the Bekenstein bound applied to black holes implies the full Penrose inequality. Indeed by setting
\begin{equation}
\mathcal{S}=\frac{k_{b}A_{e}}{4l_{p}^2},\quad\quad\quad
\mathcal{R}=\sqrt{\frac{A_e}{4\pi}},
\end{equation}
where $A_e$ is the event horizon area and $l_p=\sqrt{G\hbar/c^3}$ is the Planck length, inequality \eqref{a} with $\mathcal{E}$ denoting ADM mass yields
\begin{equation}\label{c}
\mathcal{E}^2\geq\left(\frac{c^4}{G}\sqrt{\frac{A}{16\pi}}+\sqrt{\frac{\pi}{A}}Q^2\right)^2
+\frac{4\pi c^2\mathcal{J}^2}{A},
\end{equation}
in which $A$ represents the minimum area required to enclose the horizon. The inequality is sharp in the sense that equality should be achieved only for the Kerr-Newman spacetime. The conventional derivation due to Penrose \cite{Penrose} shows that this inequality serves as a necessary condition for the weak cosmic censorship conjecture. Furthermore, results are typically stated within the framework of initial data sets for the Einstein equations that satisfy the relevant energy condition, and the inclusion of angular momentum/charge requires the absence of angular momentum/charge density outside the horizon in addition to axisymmetry (which is only needed for angular momentum). The Penrose inequality has been established in the case of maximal data by Bray \cite{Bray} and Huisken-Ilmanen \cite{HuiskenIlmanen}, and charge was added in \cite{KhuriWeinsteinYamada2,McCormick}. The inclusion of angular momentum is much more difficult and has not yet been established, although see \cite{Anglada,KhuriSokolowskyWeinstein} for partial results. Here we will establish Penrose-like inequalities involving angular momentum and charge for quasi-local masses.

The third purpose of the present work is to provide a new proof of positivity for the Wang-Yau mass, and as a result remove the spin condition used for this result in higher dimensions. The argument in \cite{WangYauPMT1} relies on the Bartnik-Shi-Tam exterior gluing construction \cite{ShiTam2002} applied to a Jang deformation \cite{SchoenYauII} of the relativistic body, where it is shown that Witten's spinor proof of the positive mass theorem \cite{Witten} is still valid and yields the desired result, even though the gluing is not smooth. Here we show how to smooth the gluing construction to allow for a conformal change to nonnegative scalar curvature, and thus avoid having to use spinors to establish positive mass.
Lastly, we will also generalize a result of Lu-Miao \cite{LuMiao} to obtain a localized version of the Penrose inequality for the static Wang-Yau mass \cite{ChenWangYau18}. In the remainder of the paper geometrized units will be used so that $c=G=1$.

\section{Statement of Main Results}
\label{sec2} \setcounter{equation}{0}
\setcounter{section}{2}

Let $\Omega$ denote an orientable connected spacelike hypersurface of a time-oriented spacetime $(N^{3,1},\mathbf{g})$.  On this hypersurface there is an induced positive definite metric $g$ and extrinsic curvature $k$, as well as induced electric and magnetic fields $E$ and $B$.  These quantities then satisfy the constraint equations
\begin{align}\label{1}
\begin{split}
16\pi \mu = R+(Tr_{g}k)^{2}-|k|_{g}^{2},&\quad\quad
8\pi J = \operatorname{div}_{g}(k-(Tr_{g}k)g),\\
4\pi \rho_e =\operatorname{div}_g E,& \quad\quad\quad\quad
4\pi \rho_b =\operatorname{div}_g B,
\end{split}
\end{align}
where $R$ is the scalar curvature, $\mu$ and $J$ represent energy and momentum density of the matter fields, and $\rho_e$, $\rho_b$ are the electric and magnetic charge density. Unless stated otherwise, $\Omega$ will be absent of charged matter so that $\rho_e =\rho_b =0$. Furthermore, the dominant energy condition $\mu\geq |J|_g$ will typically be assumed. In some cases a revised energy condition for the nonelectromagnetic matter fields will be used, namely $\mu_{EM}\geq|J_{EM}|_g$, where
\begin{equation}
\mu_{EM} = \mu-\frac{1}{8\pi}\left(|E|_{g}^{2}+|B|_{g}^{2}\right),\quad\quad\quad
J_{EM} = J +\frac{1}{4\pi}E\times B.
\end{equation}

The boundary $\partial\Omega$ of the spacelike hypersurfaces studied here will always have a single component untrapped outer piece $\Sigma$, and a possibly empty inner piece $\Sigma_h=\cup_{i=1}^{I}\Sigma_{h}^i$ consisting of apparent horizon components $\Sigma_h^i$. In addition, $\Omega$ may contain a finite number $J$ of asymptotically cylindrical ends and the union of the limiting cross-sections of these ends will be given the notation $\Sigma_{\infty}=\cup_{j=1}^{J}\Sigma_{\infty}^{j}$. The union $\Sigma_{*}=\Sigma_h \cup \Sigma_{\infty}$ will be referred to as the \textit{generalized inner boundary}. Recall that a 2-surface is untrapped if both null expansions (defined below in \eqref{nullexpansions}) are positive $\theta_{\pm}>0$, and it is referred to as a future $(+)$ or past $(-)$ apparent horizon if $\theta_{+}=0$ or $\theta_{-}=0$. In the time-symmetric case $k=0$, apparent horizons are minimal surfaces so that $H=0$, where the mean curvature $H$ is computed with respect to the unit normal $\nu$ tangent to $\Omega$ and pointing towards the outer boundary.

In 1993, Brown and York \cite{BrownYork} employed a Hamilton-Jacobi analysis to define the quasi-local energy and momentum for a spacelike 2-surface $\Sigma$ that bounds a compact spacelike hypersurface $\Omega$ in a time-oriented spacetime. Let $\sigma$ be the induced metric on $\Sigma$ then using Eulerian observers, that is unit lapse and vanishing shift, the Brown-York mass is defined by
\begin{equation}
m_{BY}(\Sigma)=\frac{1}{8\pi}\int_{\Sigma}\left(H_0-H\right)dA_{\sigma},
\end{equation}
where $H$ is the mean curvature of $\Sigma \hookrightarrow\Omega$ and $H_{0}$ is the mean curvature of an isometric embedding of $\Sigma\hookrightarrow\mathbb{R}^3$.
It is usually presumed that the Gauss curvature is positive $K_{\sigma}>0$ to obtain the existence of an isometric embedding into Euclidean 3-space, which is guaranteed by the result of Nirenberg and Pogorelov \cite{Nirenberg,Pogorelov}. Such an isometric embedding is unique up to rigid motions. Positivity of the Brown-York mass was established by Shi and Tam \cite{ShiTam2002}. Their result states the following, if $(\Omega,g)$ is compact and connected with nonnegative scalar curvature and mean convex boundary having positive Gauss curvature, then $m_{BY}(\Sigma)\geq 0$. Moreover, equality holds if and only if $(\Omega,g)$ is isometric to a domain in Euclidean space $\mathbb{R}^3$.
Recently they extend the proof to quasi positive Gauss curvature \cite{ShiTam2019}, this means that $K_{\sigma}$ is nonnegative and is positive somewhere. In \cite[Proposition 3.1]{McCormickMiao}, McCormick and Miao proved a Riemannian Penrose inequality for manifold with corners which implies a localized Riemannian Penrose inequality for the Brown-York mass. In particular, if $\Omega$ satisfies the same assumptions above and has an inner boundary $\Sigma_h$ that is an outermost minimal surface then
\begin{equation}\label{BYpenrose}
m_{BY}(\Sigma)\geq \sqrt{\frac{|\Sigma_h|}{16\pi}}.
\end{equation}
Furthermore Shi, Wang, and Yu proved a restricted rigidity statement in \cite[Theorem 2.1]{Shirigidity} for the Rimannian Penrose inequality on manifolds with corners. More precisely they show that if equality holds in \eqref{BYpenrose}, $\Sigma_h$ is strictly stable and strictly outerminimizing, for any point in $\Omega$ its minimal geodesic line to $\Sigma_h$ is contained in $\Omega$ and its distance function to $\Sigma_h$ is smooth, then $\Omega$ is a domain in the canonical slice of a Schwarzschild spacetime.

We will now consider lower bounds for the Brown-York mass in terms of charge.
The total electric and magnetic charge contained within $\Omega$ is given by
\begin{equation}
Q_e=\frac{1}{4\pi}\int_{\Sigma}g(E,\nu)dA_{\sigma},\quad\quad\quad
Q_b=\frac{1}{4\pi}\int_{\Sigma}g(B,\nu)dA_{\sigma},
\end{equation}
and the total charge squared is $Q^2=Q_e^2+Q_b^2$. The charges of individual boundary components $\Sigma_h^i$ and cylindrical end cross-sections $\Sigma_{\infty}^j$ are defined similarly, and denoted by $Q_h^i$ and $Q_{\infty}^j$ respectively. In order to state the result we will need the weak inverse mean curvature flow (IMCF) of Huisken and Ilmanen \cite{HuiskenIlmanen}. This flow $\{S_{t}\}_{t=0}^{t_0}\subset\Omega$ of 2-surfaces emanates from an inner boundary component or cylindrical end, where $t_0$ indicates the leaf of greatest area contained within $\Omega$. It may be interpreted as a flow by outermost minimal area enclosures, and as such it depends on the extension of $\Omega$ or rather the ambient manifold in which $\Omega$ resides. However, for all extensions in which the outer boundary $\Sigma$ is outerminimizing, the flow is unique and hence well-defined within $\Omega$ independent of such extensions. In this work only these type of extensions will be relevant, and we will refer to the unique flow as the \textit{indigenous IMCF} of $\Omega$.
Furthermore, the generalized inner boundary $\Sigma_*$ will be called outerminimizing if the area of any enclosing surface in $\Omega$ is not less than $|\Sigma_h|+|\Sigma_{\infty}|$.

\begin{theorem}\label{thmBYchargein}
Let $(\Omega,g)$ be a Riemannian manifold with divergence free electric and magnetic fields $E$ and $B$ satisfying $R\geq 2(|E|_{g}^2+|B|_{g}^2)$. Suppose that $\Omega$ is either compact or possesses asymptotically cylindrical ends with limiting cross-sections $\Sigma_{\infty}$, that the generalized inner boundary $\Sigma_*=\Sigma_h\cup \Sigma_{\infty}$ is strictly outerminimizing with $\Sigma_h$ the only compact minimal surface, and the single component $\Sigma$ is mean convex with positive Gauss curvature, then
\begin{equation}\label{BY1in}
m_{BY}(\Sigma)\geq \frac{1}{2} \left(\sum_{i=1}^{I}(Q_h^i)^2+\sum_{j=1}^{J}(Q_{\infty}^j)^2\right)^{\frac{1}{2}}.
\end{equation}
Moreover if $\Sigma_{*}$ has one component then
\begin{equation}\label{byq2}
m_{BY}(\Sigma)\geq \frac{|Q|}{2}+\alpha^2 \sqrt{\frac{\pi}{|\Sigma_{*}|}} Q^2,
\end{equation}
where
\begin{equation}
\alpha^2=1-\sqrt{\frac{|S_0|}{|S_{t_0}|}}
\end{equation}
and $\{S_{t}\}_{t=0}^{t_0}$ is the indigenous IMCF of $\Omega$.
\end{theorem}

\begin{remark}
The inequality \eqref{BY1in} may be interpreted as a localized version of the
positive mass theorem with charge \cite{GibbonsHawkingHorowitzPerry}, in which the ADM mass and total charge of an asymptotically flat spacetime satisfy the inequality $m\geq |Q|$. The inequality \eqref{byq2} yields a Bekenstein-like bound without angular momentum. To see this let $\mathcal{R}=\sqrt{|\Sigma|/4\pi}$ denote the area radius, then
\begin{equation}
m_{BY}(\Sigma)\geq \frac{\alpha^2}{2} \sqrt{\frac{|\Sigma|}{|\Sigma_{*}|}} \frac{Q^2}{\mathcal{R}}
\geq \frac{\alpha^2}{2}\frac{Q^2}{\mathcal{R}}
\end{equation}
where we have used the outerminimizing property to find $|\Sigma|\geq|\Sigma_*|$. The scale invariant quantity $\alpha$ may be thought of as a type of measure for the size of $\Omega$. For large domains $\alpha\sim 1$ and for small domains $\alpha\sim 0$.
\end{remark}

\begin{remark}
Based on analogy with inequalities for the ADM mass, it might be expected that saturation of \eqref{BY1in} or \eqref{byq2} implies that $(\Omega,g)$ arises from the time slice of a Majumdar-Papapetrou or Reissner-Nordstr\"om spacetime. This, however, is not the case. Indeed, observe that the Brown-York mass of a coordinate sphere $S_r$ in the Reissner-Nordstr\"om black hole with mass $m$ and charge $Q$ is given by
\begin{equation}
m_{BY}(S_r)=r\left(1-\sqrt{1-\frac{2m}{r}
+\frac{Q^2}{r^2}}\right), \quad\quad\quad
r\geq  m+\sqrt{m^2-Q^2}, \quad\quad\quad m\geq |Q|.
\end{equation}
Since this expression is decreasing in $r$ and converges to $m$ as $r\to\infty$, it follows that $m_{BY}(S_r)\geq |Q|$ so that equality in \eqref{BY1in} is not possible for nonzero charge. This is valid even in the extremal case $m=|Q|$, which coincides with a single Majumdar-Papapetrou black hole. Furthermore, in the Reissner-Nordstr\"{o}m setting $|\Sigma_*|\geq 4\pi Q^2$ with equality only for the extreme black hole. Therefore the right-hand side of \eqref{byq2} satisfies
\begin{equation}
\frac{|Q|}{2}+\alpha^2 \sqrt{\frac{\pi}{|\Sigma_{*}|}} Q^2
\leq\frac{|Q|}{2}(1+\alpha^2)=|Q|\left(1-\frac{m+\sqrt{m^2-Q^2}}{2r}\right)<|Q|,
\end{equation}
showing that saturation in \eqref{byq2} does not occur.
\end{remark}

Consider now the problem of obtaining quasi-local mass lower bounds in terms of angular momentum. Angular momentum is best behaved in the setting of axisymmetry. We will say that the spacelike hypersurface $\Omega$ is axisymmetric if there exists a $U(1)$ subgroup within the isometry group of the Riemannian manifold $(\Omega, g)$, and all relevant quantities are invariant under the $U(1)$ action. So in particular, if $\eta$ denotes the generator of the $U(1)$ symmetry then
\begin{equation}
\mathfrak{L}_{\eta}g=\mathfrak{L}_{\eta}k=\mathfrak{L}_{\eta}E=\mathfrak{L}_{\eta}B=0,
\end{equation}
where $\mathfrak{L}$ represents Lie differentiation.
A 2-surface within $\Omega$, in particular its boundary, will be referred to as axisymmetric if the generator $\eta$ is tangent to the surface at all points. A similar limiting definition may be given for axisymmetric generalized boundary components $\Sigma_{\infty}$.
In axisymmetric spacetimes all the angular momentum is contained within matter fields, since gravitational waves do not carry angular momentum. The momentum tensor $p=k-\left(Tr_{g}k\right)g$ gives rise to the angular momentum density vector $p(\eta)$, which Brown and York \cite{BrownYork} use to define the angular momentum contained within $\Omega$, namely
\begin{equation}
\mathcal{J}_{BY}(\Sigma)=\frac{1}{8\pi}\int_{\Sigma}p(\eta,\nu) dA_{\sigma}
\end{equation}
where $\nu$ is the spacelike unit normal to $\Sigma$ pointing outside of $\Omega$.
An alternate definition given by Chen-Wang-Yau \cite{ChenWangYau15} states that
\begin{equation}
\mathcal{J}(\Sigma)=-\frac{1}{8\pi}\int_{\Sigma}\langle{}^N\nabla_\eta \nu, n\rangle dA_{\sigma},
\end{equation}
where $n$ is the future-directed timelike unit normal to $\Omega$. These two definitions agree in the axisymmetric setting, see Lemma \ref{sameAM}. We will denote the angular momentum of components for the generalized boundary $\Sigma_*$ by $\mathcal{J}_{h}^i$ and $\mathcal{J}_{\infty}^j$. Lastly, in the current context it is typical to assume the maximal condition $Tr_{g}k=0$, since unlike the time-symmetric case this allows for nonzero angular momentum while at the same time giving nonnegative scalar curvature through the dominant (or other) energy condition.

\begin{theorem}\label{thmBYAMin}
Let $(\Omega,g,k)$ be a maximal axisymmetric spacelike hypersurface which satisfies the dominant energy condition $\mu\geq |J|_g$.
Suppose that $\Omega$ is either compact or possesses asymptotically cylindrical ends with limiting cross-sections $\Sigma_{\infty}$, that the generalized inner boundary $\Sigma_*=\Sigma_h\cup \Sigma_{\infty}$ is strictly outerminimizing and axisymmetric with $\Sigma_h$ the only compact minimal surface, and the single component $\Sigma$ is mean convex with positive Gauss curvature, then
\begin{equation}\label{BY2in}
m_{BY}(\Sigma)\geq \frac{1}{\sqrt{2}} \left(\sum_{i=1}^{I}|\mathcal{J}_h^i|
+\sum_{j=1}^{J}|\mathcal{J}_{\infty}^j|\right)^{\frac{1}{2}}.
\end{equation}
Moreover if $\Sigma_{*}$ has one component, and there is vanishing angular momentum density $J(\eta)=0$ then
\begin{equation}\label{byam2}
m_{BY}(\Sigma)\geq \sqrt{\frac{|\mathcal{J}|}{2}}+\frac{(2\pi \alpha)^2}{\mathcal{C}^2} \sqrt{\frac{4\pi}{|\Sigma_{*}|}} \mathcal{J}^2,
\end{equation}
where $\mathcal{C}$ is the largest circumference of all $\eta$-orbits in $\Omega$ and
\begin{equation}
\alpha^2=1-\sqrt{\frac{|S_0|}{|S_{t_0}|}}
\end{equation}
with $\{S_{t}\}_{t=0}^{t_0}$ denoting the indigenous IMCF of $\Omega$.
\end{theorem}

\begin{remark}
Inequality \eqref{BY2in} is reminiscent of the mass-angular momentum inequality established for the ADM mass in \cite{Dain0}, whereas the second inequality \eqref{byam2} implies a Bekenstein-like bound along the lines of \eqref{Bek}. To see this later statement let $\mathcal{R}_c =\mathcal{C}/2\pi$ denote the `circumference radius', then utilizing the Penrose inequality \eqref{BYpenrose} (which is known to hold also for cylindrical ends \cite{Jaracz}) and the outerminimizing property of the generalized boundary produces
\begin{equation}
m_{BY}(\Sigma)^2\geq\frac{(2\pi \alpha)^2}{\mathcal{C}^2} \sqrt{\frac{4\pi}{|\Sigma_{*}|}} \mathcal{J}^2 m_{BY}(\Sigma)
\geq\frac{(2\pi \alpha)^2}{2\mathcal{C}^2} \mathcal{J}^2
=\frac{\alpha^2}{2}\frac{\mathcal{J}^2}{\mathcal{R}_c^2} .
\end{equation}
Furthermore, as with Theorem \ref{thmBYchargein}, the naive guess that saturation of either of the inequalities \eqref{BY2in}, \eqref{byam2} implies that $\Omega$ arises from a Kerr or extreme Kerr slice is not accurate.
\end{remark}

The two previous results may be combined to produce a lower bound for quasi-local mass involving both angular momentum and charge. Furthermore, the techniques yield a Penrose-like inequality with angular momentum and charge with the same structure as \eqref{c}.

\begin{theorem}\label{combinedamch}
Under the hypotheses of Theorem \ref{thmBYAMin} with divergence free electric and magnetic fields
\begin{equation}\label{combined1}
m_{BY}(\Sigma)\geq \frac{1}{2} \left(\sum_{i=1}^{I}\sqrt{(Q_{h}^i)^4 +4(\mathcal{J}_h^i)^2}
+\sum_{j=1}^{J}\sqrt{(Q_{\infty}^j)^4+4(\mathcal{J}_{\infty}^j)^2}\right)^{\frac{1}{2}}.
\end{equation}
Moreover if $\Sigma_{*}$ has one component then
\begin{equation}\label{combined2}
m_{BY}(\Sigma)^2\geq \left(\sqrt{\frac{|\Sigma_*|}{16\pi}}
+\alpha^2 \sqrt{\frac{\pi}{|\Sigma_*|}}Q^2\right)^2+\frac{\beta^2}{2} \frac{4\pi \mathcal{J}^2}{|\Sigma_*|},
\end{equation}
where $\beta=\alpha \mathcal{R}_* \mathcal{R}_c^{-1}$ and $\mathcal{R}_*=\sqrt{|\Sigma_*|/4\pi}$.
\end{theorem}

\begin{remark}
Observe that inequality \eqref{combined2} yields the Bekenstein-like bound
\begin{equation}
m_{BY}(\Sigma)^2\geq\frac{\alpha^4}{4}\frac{Q^4}{\mathcal{R}^2}+\frac{\alpha^2}{2}
\frac{\mathcal{J}^2}{\mathcal{R}_c^2}.
\end{equation}
\end{remark}

The definition of Brown-York mass depends on the spacelike hypersurface $\Omega$. Motivated in part by a desire to remove this unwanted feature, Liu and Yau \cite{LiuYau0} proposed an alternative definition of quasi-local energy and momentum. Let $\Sigma$ be a spacelike 2-surface in spacetime $N^{3,1}$. The structure group of the normal bundle is $SO(1,1)$, so $\Sigma$ admits two smooth non-vanishing future directed null normal
vector fields $l_+$ and $-l_-$ such that $\langle l_+,l_-\rangle=2$ which are unique up to positive rescaling. In particular, when $\Sigma$ is the outer boundary of $\Omega$ we may write $l_{\pm}=\nu\pm n$ where as before $n$ and $\nu$ are the unit timelike and spacelike normal vectors to the surface $\Sigma$ with respect to $\Omega$. The null mean curvatures then take the form
\begin{equation}\label{nullexpansions}
\theta_\pm=\text{div}_{\Sigma}l_{\pm}=H\pm Tr_{\Sigma}k,
\end{equation}
and the mean curvature vector with norm squared is given by
\begin{equation}
\vec{H}=\frac{1}{2}\left(\theta_- l_+ +\theta_+ l_-\right)=H\nu-(Tr_{\Sigma}k)n,\quad\quad\quad
|\vec{H}|^2=\theta_+\theta_-=H^2-(Tr_{\Sigma}k)^2.
\end{equation}
Assume that $(\Sigma,\sigma)$ is of positive Gauss curvature so that it admits an isometric embedding into the Euclidean time slice of Minkowski space $\mathbb{R}^{3,1}$. If $\vec{H}_0 =H_0 \nu_0$ denotes the corresponding mean curvature vector of this embedding and the spacetime mean curvature vector is non-timelike $|\vec{H}|^2\geq 0$, then the Liu-Yau energy is defined as
\begin{equation}
m_{LY}(\Sigma)=\frac{1}{8\pi}\int_{\Sigma}\left(|\vec{H}_0|-|\vec{H}|\right)dA_\sigma
=\frac{1}{8\pi}\int_{\Sigma}\left(H_0-\sqrt{H^2-(Tr_{\Sigma}k)^2}\right)dA_\sigma,
\end{equation}
In \cite{LiuYau0,LiuYau} they prove that in the setting above, with the dominant energy condition $\mu\geq |J|_g$ valid on $\Omega$, the positivity property $m_{LY}(\Sigma)\geq 0$ holds and if equality is achieved then $(\Omega,g,k)$ must arise from an embedding into Minkowski space.

The Liu-Yau angular momentum is obtained as a decomposition of the Brown-York momentum surface density 1-form $p(\nu)^{T}$, where the superscript $T$ denotes projection onto the tangent space of $\Sigma$. Since $\Sigma$ is topologically a 2-sphere, by Hodge decomposition there exist two functions $\upsilon$ and $\varpi$ on $\Sigma$ which are unique up to a constant, such that $p(\nu)^{T}=d\upsilon +*_{\sigma} d\varpi$ where $*_\sigma$ is the Hodge star operator on $\Sigma$. Indeed, they may be obtained by solving
\begin{equation}
\Delta_{\sigma} \upsilon=\delta p(\nu)^{T},\quad\quad\quad \Delta_{\sigma} \varpi =-\star_{\sigma} d p(\nu)^{T},
\end{equation}
where $\delta$ is the codifferential. It may be checked that $d p(\nu)^{T}$ is independent of the choice of frame $\{\nu,n\}$ for the normal bundle of $\Sigma$. Hence $\mathbf{j}=*_{\sigma} d\varpi$ does not depend on $\Omega$, in contrast to $d\upsilon$. If $\eta^T$ is the tangential part, to the embedding of $\Sigma$ in $\mathbb{R}^3$, of an axisymmetric Killing field then the Liu-Yau angular momentum is
\begin{equation}
\mathcal{J}_{LY}(\Sigma)=\frac{1}{8\pi}\int_{\Sigma}\mathbf{j}(\eta^T) dA_\sigma.
\end{equation}
In Lemma \ref{sameAM} below we show that if $\Sigma$ is axisymmetric then this definition agrees with the previous one $\mathcal{J}_{LY}(\Sigma)=\mathcal{J}(\Sigma)$.

\begin{theorem}\label{thm2.7}
Let $(\Omega,g,k)$ be an initial data set for the Einstein equations satisfying the dominant energy condition $\mu\geq |J|_g$ which is strict on horizons. Suppose that $\Omega$ is compact with boundary consisting of a disjoint union $\partial\Omega=\Sigma_h \cup\Sigma$ where $\Sigma_h$ is a (nonempty) apparent horizon, no other closed apparent horizons are present, and the single component $\Sigma$ is untrapped with positive Gauss curvature, then there exists a nonzero positive constant $\gamma$ independent of horizon area such that
\begin{equation}\label{LY1in}
m_{LY}(\Sigma)\geq \frac{\gamma}{1+\gamma}\sqrt{\frac{|\Sigma_{h}|}{4\pi}}.
\end{equation}
Moreover, if in addition the enhanced energy condition $\mu_{EM}\geq|J_{EM}|_g$ holds, $(\Omega,g,k,E,B)$ and $\partial\Omega$ are axially symmetric, and $\Sigma_h$ is stable in the sense of apparent horizons then
\begin{equation}\label{LY2in}
m_{LY}(\Sigma)\geq \frac{\gamma}{1+\gamma}
\left(\sum_{i=1}^I \sqrt{(Q_{h}^i)^4+4(\mathcal{J}_{h}^i)^2}\right)^{\frac{1}{2}}.
\end{equation}
\end{theorem}

\begin{remark}
Inequality \eqref{LY1in} is a Penrose-like inequality for the Liu-Yau energy, whereas \eqref{LY2in} is akin to a localized version of the ADM mass-angular momentum-charge inequality which has so far only been established in the maximal case \cite{ChruscielLiWeinstein,Dain0,KhuriWeinstein,SchoenZhou}. Furthermore, inequality \eqref{LY2in} holds without the assumption of axisymmetry if the angular momentum contribution from the right-hand side is dropped, while the stability hypothesis can be removed if $\Sigma_h$ has only one component \cite{AnderssonMetzger,Eichmair}. The constant $\gamma$ is invariant under rescalings of the metric and hence independent of $|\Sigma_h|$. Defined in \eqref{definitiongamma}, it is based on a Dirichlet energy and may be thought of as a type of capacity associated with $\Omega$. Lastly the hypothesis of a strict dominant energy condition on horizons may be removed in many circumstances by using the localized perturbations of Corvino and Huang \cite{CorvinoHuang}.
\end{remark}

The above result can be improved in the case when the horizon has a single component to obtain a localized Penrose-like inequality in the spirit of \eqref{c}, since the square of the $\eta$-orbit circumference $\mathcal{C}^2$ has the units of area. We will make use of the condition
\begin{equation}\label{frobenius}
\eta\wedge d\eta=0,
\end{equation}
which from Frobenius' theorem guarantees that the Killing field $\eta$ is hypersurface orthogonal.

\begin{theorem}\label{thm2.9}
Let $(\Omega,g,k,E,B)$ be an axisymmetric initial data set for the Einstein-Maxwell equations satisfying \eqref{frobenius}, $J(\eta)=0$, and the energy condition $\mu_{EM}\geq |J|_g$ which is strict on horizons. Suppose that $\Omega$ is compact with axisymmetric boundary consisting of a disjoint union $\partial\Omega=\Sigma_h \cup\Sigma$ where $\Sigma_h$ is a single component (nonempty) apparent horizon, no other closed apparent horizons are present, and the single component $\Sigma$ is untrapped with positive Gauss curvature, then there exists a nonzero positive constant $\lambda$ independent of horizon area such that
\begin{equation}\label{LY3in}
m_{LY}(\Sigma)^2\geq \left(\frac{\gamma}{1+\gamma}\sqrt{\frac{|\Sigma_{h}|}{4\pi}}
+\lambda \sqrt{\frac{\pi}{|\Sigma_h|}}Q^2\right)^2
+\frac{\lambda\gamma}{1+\gamma}\frac{8\pi^2\mathcal{J}^2}{\mathcal{C}^2},
\end{equation}
where $\gamma$ is as in Theorem \ref{thm2.7}. Moreover, the same inequality holds without the assumption of axisymmetry if the contribution from angular momentum is removed.
\end{theorem}

\begin{remark}
The constant $\lambda$ given in \eqref{lambdadefinition} is invariant under rescalings of the metric, and hence is independent of horizon area. It is defined via the largest area in an inverse mean curvature flow within $\Omega$. Furthermore inequality \eqref{LY3in} implies a Bekenstein-like bound
\begin{equation}
m_{LY}(\Sigma)^2\geq
\lambda_{*}^2 \frac{Q^4}{4\mathcal{R}^2}
+\frac{2\lambda\gamma}{1+\gamma}\frac{\mathcal{J}^2}{\mathcal{R}_c^2},
\end{equation}
where $\mathcal{R}$ is the area radius of $\Sigma$, $\mathcal{R}_c$ is the circumference radius of $\Omega$, and $\lambda_*=\lambda\sqrt{|\Sigma|/|\Sigma_h|}$.
\end{remark}

O'Murchadha, Szabados, and Tod \cite{Murchadha} showed that there are examples of spacelike 2-surfaces $\Sigma\subset \mathbb{R}^{3,1}$ with positive Gauss curvature and spacelike mean curvature vector, but with $m_{LY}(\Sigma)>0$. This shows that the Liu-Yau mass/energy is `too positive', because an optimal quasi-local mass should vanish for surfaces in Minkowski space. In order to rectify this issue, Wang and Yau introduced a new definition of quasi-local mass/energy in \cite{WangYauPMT,WangYauPMT1} using Hamilton-Jacobi analysis and the notion of optimal isometric embeddings.

Let $\Sigma\hookrightarrow N^{3,1}$ be a compact spacelike 2-surface in spacetime with induced metric $\sigma$, and let $\{e_3,e_4\}$ be an orthonormal frame for its normal bundle where $e_3$ is spacelike and $e_4$ is future-directed timelike. Then the connection 1-form of the normal bundle in this gauge is given by
\begin{equation}
\alpha_{e_3}(\cdot)=\langle{}^N\nabla_{(\cdot)}e_3,e_4\rangle.
\end{equation}
The \textit{Wang-Yau data set} consists of $\Sigma$ together with the triple $(\sigma,|\vec{H}|, \alpha)$. Now let $\tau$ be a time function on $\Sigma$ satisfying the convexity condition
\begin{equation}\label{convexitycondition}
\left(1+|\nabla\tau|^2\right)K_{\hat{\sigma}}=K_{\Sigma}
+\left(1+|\nabla\tau|^2\right)^{-1}\text{det}(\nabla^2\tau)>0,
\end{equation}
where $K_{\hat{\sigma}}$ is the Gauss curvature of $\hat{\sigma}=\sigma+d\tau^2$. By the Nirenberg/Pogorelov theorem there exists a unique (up to rigid motion) isometric embedding into $\mathbb{R}^3$, and from this one obtains an isometric embedding
$\iota:\Sigma\hookrightarrow \mathbb{R}^{3,1}$. The time function on $\iota(\Sigma)$ is then given by $\tau=-\langle T_0,\iota\rangle$, where $T_0$ is the designated future timelike unit Killing field on $\mathbb{R}^{3,1}$. The projection of the embedding to the $\mathbb{R}^3$ which is orthogonal to $T_0$ will be denoted by $(\hat{\Sigma},\hat{\sigma})$, and its mean curvature with respect to the outer normal will be labeled $\hat{H}_0$; the mean curvature vector of $\iota(\Sigma)$ will be denoted $\vec{H}_0$.

Let $\{\bar{e}_3,\bar{e}_4\}$ be the unique orthonormal frame for the normal bundle of $\Sigma$ in $N^{3,1}$ such that $\bar{e}_3$ is spacelike, $\bar{e}_4$ is future-directed timelike, and
\begin{equation}\label{con1}
\langle\vec{H},\bar{e}_3\rangle >0,\quad\quad\quad \langle\vec{H},\bar{e}_4\rangle
=\frac{-\Delta\tau}{\sqrt{1+|\nabla\tau|^2}}.
\end{equation}
The Wang-Yau energy with respect to the observer determined by the embedding and time function is defined to be
\begin{equation}\label{wangyauenergy}
E_{WY}(\Sigma,\iota,\tau)
=\frac{1}{8\pi}\int_{\Sigma}\left(\mathfrak{H}_0-\mathfrak{H}\right)dA_\sigma,
\end{equation}
where the \textit{generalized mean curvature} is
\begin{equation}
\mathfrak{H}=\sqrt{1+|\nabla\tau|^2}\langle\vec{H},\bar{e}_3\rangle
-\alpha_{\bar{e}_3}(\nabla\tau),
\end{equation}
and if $\vec{H}_0$ is spacelike then $\mathfrak{H}_0=\sqrt{1+|\nabla\tau|^2}\hat{H}_0$.
If $\tau=0$ it is clear that this definition agrees with the Liu-Yau energy. Moreover if $\Sigma$ lies in Minkowski space then $\mathfrak{H}_0=\mathfrak{H}$, so that the Wang-Yau energy vanishes. This rectifies the issue found in \cite{Murchadha} for the Liu-Yau energy. It should be pointed out that the choice of gauge $\{\bar{e}_3,\bar{e}_4\}$ for the normal bundle minimizes the surface Hamiltonian \cite[Section 2]{WangYauPMT1}. Namely, if $\{e_3,e_4\}$ is any other oriented orthonormal frame of the normal bundle, and $\vec{H}$ is spacelike then
\begin{equation}\label{genmean}
\int_{\Sigma}\mathfrak{H}(e_3,e_4) dA_{\sigma}
\geq\int_{\Sigma}\mathfrak{H}(\bar{e}_3,\bar{e}_4) dA_{\sigma}.
\end{equation}

The surface $\Sigma$ and time function $\tau$ are said to be \textit{admissible} if the convexity condition \eqref{convexitycondition} is satisfied, $\Sigma$ arises as the untrapped boundary of a spacelike hypersurface $(\Omega,g,k)\hookrightarrow N^{3,1}$, and the generalized mean curvature is positive $\mathfrak{H}(e_{3}',e_{4}')>0$ for the normal bundle frame $\{e_{3}',e_{4}'\}$ determined by the solution of Jang's equation, see Definition 5.1 of \cite{WangYauPMT1}. For admissible sets it is proven in \cite{WangYauPMT, WangYauPMT1} that the energy is always nonnegative $E_{WY}(\Sigma,\iota,\tau)\geq 0$. Furthermore, if the energy vanishes then the initial data $(\Omega,g,k)$ arises from an embedding into Minkowski space.

In analogy with special relativity, the Wang-Yau mass is defined as the
infimum of energy over all admissible observers $(\iota,\tau)$. To facilitate this
the energy may be rewritten as
\begin{equation}
E_{WY}(\Sigma,\iota,\tau)
=\frac{1}{8\pi}\int_{\Sigma}\left(\varrho +\sigma(\mathfrak{j},\nabla\tau)\right)dA_\sigma,
\end{equation}
where
\begin{equation}
\varrho=\frac{\sqrt{|\vec{H}_0|^2+\frac{(\Delta\tau)^2}{1+|\nabla\tau|^2}}
-\sqrt{|\vec{H}|^2+\frac{(\Delta\tau)^2}{1+|\nabla\tau|^2}}}{\sqrt{1+|\nabla\tau|^2}},  \quad\quad \mathfrak{j}=\varrho\nabla\tau
-\nabla\sinh^{-1}\left(\frac{\varrho\Delta\tau}{|\vec{H}_0||\vec{H}|}\right)
+\alpha_{\vec{H}_0}-\alpha_{\vec{H}},
\end{equation}
and $\vec{H}_0$ is the mean curvature vector associated with $\iota(\Sigma)\subset\mathbb{R}^{3,1}$ which is assumed to be spacelike. The Euler-Lagrange equation then becomes
\begin{equation}
\text{div}_{\sigma}\mathfrak{j}=0,
\end{equation}
and when coupled with the isometric embedding equation
\begin{equation}
\langle d\iota,d\iota\rangle_{\mathbb{R}^{3,1}}=\sigma,
\end{equation}
this becomes the \textit{optimal isometric embedding system}. Thus,
a critical point of the Wang-Yau energy corresponds to an optimal isometric embedding. If the infimum is achieved then the Wang-Yau mass is given by
\begin{equation}
m_{WY}(\Sigma)=\frac{1}{8\pi}\int_{\Sigma}\varrho dA_\sigma.
\end{equation}
In general the optimal isometric embedding may not be unique, but as shown in \cite{ChenWangYau14} it is unique locally if $\varrho>0$. Furthermore, nonnegativity of the energy as discussed above implies nonnegativity of the mass, and the mass is zero for surfaces in Minkowski space.

The optimal isometric embedding gives the `best match' to the physical surface data $(\Sigma,\sigma,|\vec{H}|,\alpha)$. This optimal reference surface can then be used to define other conserved quantities such as quasi-local angular momentum \cite{ChenWangYau15}. Let $\eta$ denote a Killing field generating the axisymmetry in $\mathbb{R}^{3,1}$, then the Chen-Wang-Yau angular momentum \cite{ChenWangYau13} for an optimal isometric embedding is
\begin{equation}
\mathcal{J}_{CWY}(\Sigma,\iota, \tau)=\frac{1}{8\pi}\int_{\Sigma}\left(\varrho\langle\eta,T_0\rangle
+\mathfrak{j}(\eta^T)\right)dA_\sigma.
\end{equation}
As shown in \cite[Proposition 6.1]{ChenWangYau15}, if the spacetime $N^{3,1}$ and surface are axially symmetric then
\begin{equation}
\mathcal{J}_{CWY}(\Sigma,\iota,\tau)=\mathcal{J}(\Sigma).
\end{equation}
Moreover, it is shown in \cite[Theorem 3]{ChenWangYau14} that if the surface $\Sigma$ has positive Gauss curvature, and $\tau=0$ is a solution of the optimal isometric embedding equation with $|\vec{H}_0|>|\vec{H}|>0$, then $\tau=0$ minimizes the Wang-Yau energy among axisymmetric time functions with $K_{\hat{\sigma}}>0$.
It is possible then in this situation that the Liu-Yau energy is the minimizer of Wang-Yau energy, in which case Theorems \ref{thm2.7} and \ref{thm2.9} would also hold for the Wang-Yau mass. With a different approach we are able to establish such lower bounds for general admissible $\tau$.

\begin{theorem}\label{thm2.11}
Let $(\Sigma,\sigma,|\vec{H}|,\alpha)$ be a Wang-Yau data set arising from a spacetime $N^{3,1}$, with optimal isometric embedding $(\iota,\tau)$. Assume that $\Sigma$ and $\tau$ are admissible, and let $(\Omega,g,k)\hookrightarrow N^{3,1}$ be the bounding spacelike hypersurface. Suppose that
$(\Omega,g,k)$ satisfies the dominant energy condition $\mu\geq |J|_g$ which is strict on horizons, is compact with boundary consisting of a disjoint union $\partial\Omega=\Sigma_h \cup\Sigma$ where $\Sigma_h$ is a (nonempty) apparent horizon, and no other closed apparent horizons are present, then there exists a nonzero positive constant $\gamma$ independent of horizon area such that
\begin{equation}\label{WY1in}
m_{WY}(\Sigma)\geq \frac{\gamma}{1+\gamma}\sqrt{\frac{|\Sigma_{h}|}{4\pi}}.
\end{equation}
Moreover, if in addition the enhanced energy condition $\mu_{EM}\geq|J_{EM}|_g$ holds, $(\Omega,g,k,E,B)$ and $\partial\Omega$ are axially symmetric, and $\Sigma_h$ is stable in the sense of apparent horizons then
\begin{equation}\label{WY2in}
m_{WY}(\Sigma)\geq \frac{\gamma}{1+\gamma}
\left(\sum_{i=1}^I \sqrt{(Q_{h}^i)^4+4(\mathcal{J}_{h}^i)^2}\right)^{\frac{1}{2}}.
\end{equation}
\end{theorem}

The constant $\gamma$ here, as well as that of $\lambda$ in the next result, are defined analogously with those of Theorems \ref{thm2.7} and \ref{thm2.9}. Furthermore, as with the Liu-Yau energy bounds this
can be improved in the case when the horizon has a single component to obtain a localized Penrose-like inequality in the spirit of \eqref{c}.

\begin{theorem}\label{thm2.12}
Let $(\Sigma,\sigma,|\vec{H}|,\alpha)$ be a Wang-Yau data set arising from a spacetime $N^{3,1}$, with optimal isometric embedding $(\iota,\tau)$. Assume that $\Sigma$ and $\tau$ are admissible, and let $(\Omega,g,k,E,B)\hookrightarrow N^{3,1}$ be the bounding spacelike hypersurface. Suppose that
$(\Omega,g,k,E,B)$ is axisymmetric, satisfies \eqref{frobenius} and $J(\eta)=0$ along with the energy condition $\mu_{EM}\geq |J|_g$ which is strict on horizons, is compact with axisymmetric boundary consisting of a disjoint union $\partial\Omega=\Sigma_h \cup\Sigma$ where $\Sigma_h$ is a single component (nonempty) apparent horizon, and no other closed apparent horizons are present,  then there exists a nonzero positive constant $\lambda$ independent of horizon area such that
\begin{equation}\label{WY3in}
m_{WY}(\Sigma)^2\geq \left(\frac{\gamma}{1+\gamma}\sqrt{\frac{|\Sigma_{h}|}{4\pi}}
+\lambda \sqrt{\frac{\pi}{|\Sigma_h|}}Q^2\right)^2
+\frac{\lambda\gamma}{1+\gamma}\frac{8\pi^2\mathcal{J}^2}{\mathcal{C}^2}.
\end{equation}
Moreover, the same inequality holds without the assumption of axisymmetry if the contribution from angular momentum is removed.
\end{theorem}

\section{Review of the Proofs of Shi-Tam, Liu-Yau, and Wang-Yau}
\label{sec3} \setcounter{equation}{0}
\setcounter{section}{3}

Various elements of the proofs of positivity for the three quasi-local masses will be utilized to establish the results stated in the previous section. Here
we outline the proofs and discuss the primary tools that will be employed later.\medskip

\noindent\textbf{The Shi-Tam Proof \cite{ShiTam2002}.} Let $(\Omega, g)$ be a Riemannian manifold with nonnegative scalar curvature and single component boundary $\partial\Omega=\Sigma$ of positive Gauss $K$ and mean curvature $H$. The arguments below easily work for a boundary with several components, however for simplicity of discussion we restrict to a single component. By the Weyl embedding theorem $\Sigma$ isometrically embeds uniquely in $\mathbb{R}^3$. The image will be labeled $\Sigma_0$ with Gauss curvature $K_0=K$ and mean curvature $H_0>0$. Consider the unit normal flow
$\{\Sigma_r\}_{r\in[0,\infty)}$ emanating outward from $\Sigma_0$, so that the
Euclidean metric on the exterior region $M_+=[0,\infty)\times \Sigma_{0}$ is given by $dr^2+\sigma_r$, where $\sigma_r$ is the induced metric on the leaf $\Sigma_r$.
Construct now a new (asymptotically flat) metric
\begin{equation}
g_+=u^2dr^2+\sigma_r
\end{equation}
on $M_+$, in which the function $u:M_+\to\mathbb{R}_+$ satisfies the parabolic initial value problem
\begin{equation}\label{parabolic1}
\begin{cases}
H_r\frac{\partial u}{\partial r}=u^2\Delta_r u+K_{r}\left(u-u^3\right)& \text{on } \text{ }\text{ }M_+\\
u(0,x)=u_{0}(x)
\end{cases}.
\end{equation}
Here $\Delta_{r}$, $K_{r}$, and $H_r$ are the Laplacian, Gauss, and mean curvatures with respect to $\sigma_r$. This equation guarantees that the scalar curvature of $g_+$ vanishes $R_{g_+}=0$, and if $u_0=H_0/H$ then the mean curvature of $\partial M_+$ is $H$. Furthermore, since the Gauss curvatures $K_r$ are positive the function of $r$ given by
\begin{equation}\label{MS2}
\frac{1}{8\pi}\int_{\Sigma_r}H_r\left(1-u^{-1}\right)dA_{\sigma_r}
\end{equation}
is nonincreasing. It also converges to the ADM mass $\mathbf{m}$ of $g_+$ as $r\rightarrow\infty$ and agrees with the Brown-York mass of $\Sigma$ at $r=0$, hence
\begin{equation}
m_{BY}(\Sigma)\geq \mathbf{m}.
\end{equation}
It remains to show that the ADM mass is nonnegative. This, however, follows
due to the fact that there is a nonnegative scalar curvature fill-in for $(M_+,g_+)$, namely $(\Omega,g)$. More precisely, the positive mass theorem holds \cite{Miao,ShiTam2002} for the composite manifold $(\Omega\cup M_+, g\cup g_+)$ as both sides are of nonnegative scalar curvature, and the induced metrics as well as the mean curvatures agree along the gluing surface $\Sigma$.
\medskip

\noindent\textbf{The Liu-Yau Proof \cite{LiuYau0,LiuYau}.}
Consider a compact initial data set $(\Omega,g,k)$. Even when this satisfies the dominant energy condition $\mu\geq|J|_g$, it may not have nonnegative scalar curvature and this is an impediment to applying the techniques of Shi-Tam discussed above. Thus, the idea is to deform to nonnegative scalar curvature while preserving the induced metric on the boundary. This may be achieved, as in the case of the spacetime positive mass theorem \cite{SchoenYauII}, through a two step procedure. The first step is the Jang deformation $g\rightarrow \bar{g}=g+df^2$ where $f:\Omega\rightarrow\mathbb{R}$ is a solution of the Jang equation
\begin{equation}\label{Jang}
\begin{cases}
\left(g^{ij}-\frac{f^if^{j}}{1+|\nabla f|^2_g}\right)\left(\frac{\nabla_{ij}f}{\sqrt{1+|\nabla f|^2_g}}-k_{ij}\right)=0 & \text{in $\Omega$}\\
f=\tau & \text{on $\partial\Omega$}
\end{cases},
\end{equation}
with $f^i=g^{ij}f_j$, the covariant derivative $\nabla$ is with respect to $g$, and $\tau=0$. The boundary condition guarantees that the induced metrics from $g$ and the Jang metric $\bar{g}$ agree. Moreover the equation implies that the scalar curvature of the Jang metric is weakly nonnegative, that is
\begin{equation}\label{Jangscalar}
 \bar{R}=2\left(\mu-J(w)\right)+|h-k|^2_{\bar{g}}
 +2|X|_{\bar{g}}^2-2\text{div}_{\bar{g}}X\geq 2|X|_{\bar{g}}-2\text{div}_{\bar{g}}X,
\end{equation}
where $h$ is second fundamental form of the graph $t=f(x)$ in the product manifold $(\Omega\times\mathbb{R},g+dt^2)$, and $w$, $X$ are 1-forms given by
\begin{equation}\label{def-h-w-X}
 h_{ij}=\frac{\nabla_{ij}f}{\sqrt{1+|\nabla f|^2_g}},\qquad w_i=\frac{f_i}{\sqrt{1+|\nabla f|^2_g}},\qquad X_i=\frac{f^j}{\sqrt{1+|\nabla f|^2_g}}\left(h_{ij}-k_{ij}\right).
\end{equation}
The Jang scalar curvature has a sufficient nonnegativity property that it allows for a solution $u>0$ of the zero scalar curvature equation
\begin{equation}\label{zeroscalarcurvature}
\begin{cases}
\Delta_{\bar{g}}u-\frac{1}{8}\bar{R}u=0 & \text{in $\Omega$}\\
u=1 & \text{on $\partial\Omega$}
\end{cases}.
\end{equation}
The conformal metric $\tilde{g}=u^4 \bar{g}$ then has zero scalar curvature and its induced boundary metric agrees with that of $g$. The mean curvature of the boundary $\tilde{H}=\bar{H}+4\partial_{\bar{\nu}} u$ with respect to the conformal metric may not be positive but it does satisfy
\begin{equation}\label{h1}
\int_{\partial\Omega}\tilde{H} dA_{\sigma}\geq \int_{\partial\Omega}\left(\bar{H}-X(\bar{\nu})\right)dA_{\sigma},
\end{equation}
and a computation \cite[Section 5]{Yau} shows that
\begin{equation}\label{h2}
\bar{H}-X(\bar{\nu})\geq |\vec{H}|>0,
\end{equation}
where $\bar{H}$ and $\bar{\nu}$ are the mean curvature and unit normal of the boundary with respect to $\bar{g}$. Thus, assuming that the boundary Gauss curvature is positive, one may construct an asymptotically flat outer manifold $(M_+,\tilde{g}_+)$ of zero scalar curvature following the Shi-Tam approach with mean curvature at $\partial M_+$ given by $\tilde{H}_+ = \bar{H}-X(\bar{\nu})$. Although the mean curvature $\tilde{H}$ of the inner manifold $(\Omega,\tilde{g})$ does not necessarily agree with that of the outer manifold, it is shown nevertheless that the ADM mass $\tilde{\mathbf{m}}$ of the composite $(\Omega\cup M_+, \tilde{g}\cup \tilde{g}_+)$ is nonnegative using Witten's spinor proof. Therefore, as in the Shi-Tam proof
\begin{equation}\label{LYadm}
m_{LY}(\Sigma)\geq \frac{1}{8\pi}\int_{\Sigma}\left(H-\left(\bar{H}-X(\bar{\nu})\right)\right)dA_{\sigma}
\geq \tilde{\mathbf{m}}\geq 0.
\end{equation}

\noindent\textbf{The Wang-Yau Proof \cite{WangYauPMT,WangYauPMT1}.}
The Wang-Yau proof is an extension of the Liu-Yau method to the case when $\tau$ is nontrivial, although it forgoes the second step of conformal deformation.
Let $(\Sigma,\sigma,|\vec{H}|,\alpha)$ be a Wang-Yau data set. If the time function $\tau$ is admissible with $\Sigma$ then there is a compact spacelike hypersurface
$(\Omega,g,k)\hookrightarrow N^{3,1}$ such that $\partial\Omega=\Sigma$. Admissibility also guarantees that the projection $(\hat{\Sigma},\hat{\sigma})$ to $\mathbb{R}^3$, of the isometric embedding $\iota:\Sigma\hookrightarrow\mathbb{R}^{3,1}$, has positive Gauss curvature $K_{\hat{\sigma}}>0$, and in addition that the generalized mean curvature is positive
\begin{equation}
\mathfrak{H}(e_3',e_4')>0,\quad\quad e_3'=\cosh\psi \nu +\sinh\psi n,\quad\quad
\sinh\psi=\frac{\nu(f)}{\sqrt{1+|\nabla f|_g^2}},
\end{equation}
where $f$ is a solution of the Jang-Dirichlet problem \eqref{Jang}. Using the projection to Euclidean 3-space one may then construct an asymptotically flat zero scalar curvature outer manifold $(M_+,\bar{g}_+)$ via the Shi-Tam procedure, with the mean curvature of $\partial M_+$ given by $(1+|\nabla\tau|^2)^{-\frac{1}{2}}\mathfrak{H}(e_3',e_4')$. A computation shows that this agrees with the mean curvature quantity from the Jang equation, namely
\begin{equation}
\frac{\mathfrak{H}(e_3',e_4')}{\sqrt{1+|\nabla\tau|^2}}=\bar{H}-X(\bar{\nu}).
\end{equation}
This observation leads to nonnegativity of the ADM mass $\bar{\mathbf{m}}$ for the glued manifold $(\Omega\cup M_+,\bar{g}\cup \bar{g}_+)$, again by following Witten's spinor argument. It then follows from the Shi-Tam monotonicity that
\begin{equation}
E_{WY}(\Sigma,\iota,\tau)\geq \frac{1}{8\pi}\int_{\hat{\Sigma}}\left(\hat{H}_0
-\frac{\mathfrak{H}(e_3',e_4')}{\sqrt{1+|\nabla\tau|^2}}\right)dA_{\hat{\sigma}}\geq\bar{\mathbf{m}}\geq 0,
\end{equation}
where in the first inequality on the left \eqref{genmean} was used. It immediately follows that the mass, as the infimum of energy, is also nonnegative.

\section{Brown-York Mass, Angular Momentum, and Charge Inequalities}
\label{sec4}\setcounter{equation}{0}
\setcounter{section}{4}

The four definitions, discussed in Section \ref{sec2}, of angular momentum of the 2-surface $\Sigma$ all agree in the axisymmetric regime, and correspond to the associated Komar integral
\begin{equation}
\mathcal{J}_{K}(\Sigma)=-\frac{1}{8\pi}\int_{\Sigma}\star_{N} d\eta.
\end{equation}
Most of the statements in the next lemma may be found in various places throughout the literature, with assorted hypotheses. We combine them here in one result (with minimal hypotheses) for convenience, and bring them into the setting of the current paper.

\begin{lemma}\label{sameAM} Let $\Sigma$ be an axially symmetric spacelike 2-surface in spacetime $N^{3,1}$, and denote the rotational Killing field by $\eta$.
\begin{enumerate}
\item If $\Sigma$ arises as the boundary of a spacelike hypersurface $(\Omega,g,k)\hookrightarrow N^{3,1}$, and $p(\nu)^T$ is invariant under the action of $\eta$, then $\mathcal{J}_{BY}(\Sigma)=\mathcal{J}_{LY}(\Sigma)=\mathcal{J}(\Sigma)$.
    If in addition the spacetime is axisymmetric, then these quantities agree with $\mathcal{J}_{K}(\Sigma)$.

	\item Let $\iota:\Sigma \hookrightarrow \mathbb{R}^{3,1}$ be an axisymmetric optimal isometric embedding with time function $\tau=-\langle T_0,\iota \rangle$. If $N^{3,1}$ is axisymmetric then $\mathcal{J}_{CWY}(\Sigma,\iota,\tau)=\mathcal{J}(\Sigma)=\mathcal{J}_{K}(\Sigma)$.
\end{enumerate}
\end{lemma}

\begin{proof}
Since $\Sigma$ is axisymmetric, $\eta^T=\eta$. Thus on $\Sigma$ we have
\begin{equation}
p(\eta,\nu)=k(\eta,\nu)=\langle{}^N\nabla_{\eta}n,\nu\rangle
=-\langle n,{}^{N}\nabla_{\eta}\nu \rangle,
\end{equation}
and integrating produces
$\mathcal{J}_{BY}(\Sigma)=\mathcal{J}(\Sigma)$. Furthermore, the invariance of $p(\nu)^T$ under the action of $\eta$ implies that $\eta(\upsilon)=0$. Hence
\begin{equation}
\mathbf{j}(\eta^T)=\mathbf{j}(\eta)=p(\eta,\nu),
\end{equation}
and integrating produces $\mathcal{J}_{LY}(\Sigma)=\mathcal{J}_{BY}(\Sigma)$. In order to compute the Komar integrand let $\epsilon$ be the volume form of $N^{3,1}$ and set $\epsilon_{\sigma}=dA_{\sigma}$. Then in local coordinates on $\Sigma$ we have
\begin{equation}
\left(\star_{N} d\eta\right)_{ab}
=\epsilon_{abcd}{}^{N}\nabla^{[c} \eta^{d]}=\frac{1}{2}\left(\mathfrak{i}_{n}\mathfrak{i}_{\nu}\epsilon\right)_{ab}
\left({}^{N}\nabla^{\nu}\eta^n-{}^{N}\nabla^n \eta^{\nu}\right)
=\langle{}^{N}\nabla_{\nu}\eta,n\rangle \left(\epsilon_{\sigma}\right)_{ab},
\end{equation}
where $\mathfrak{i}$ denotes interior product. Again the desired conclusion follows by integrating over $\Sigma$, and this concludes the proof of (1).

The statement (2) is proven in \cite[Proposition 6.1]{ChenWangYau15}. To summarize, since the isometric embedding $\iota$ is axisymmetric, the time function $\tau$ is axisymmetric which implies that $\eta$ is a symmetry of the projection $\hat{\Sigma}\subset\mathbb{R}^3$. It follows that $\langle\eta,T_0\rangle=0$. Moreover the first three terms of $\mathfrak{j}(\eta)$ vanish immediately according to the hypotheses. Lastly $\alpha_{\vec{H}}(\eta)$ may be expressed as $\langle{}^{N}\nabla_{\eta}e_3,e_4\rangle$ for any oriented frame $\{e_3,e_4\}$ of the normal bundle, since $\eta$ is a Killing field. This shows that $\mathcal{J}_{CWY}(\Sigma,\iota,\tau)=\mathcal{J}(\Sigma)$.
\end{proof}

The goal of this section is to establish Theorems \ref{thmBYchargein}, \ref{thmBYAMin}, and \ref{combinedamch}. One ingredient of the proof consists of an inequality relating horizon area to charge and angular momentum. This class of inequalities is motivated in part by black hole thermodynamics, particularly the desire for a nonnegative black hole temperature, and have been proven in generality
for stable apparent horizons \cite{BrydenKhuri,Dain2012,ClementJaramilloReiris,ClementReirisSimon}. Here we extend this result to the setting of asymptotically cylindrical ends. Let us first recall the relevant definitions. An \textit{asymptotically cylindrical end} within an initial data set $(\Omega,g,k,E,B)$ is a subset diffeomorphic to $[1,\infty)\times S^2$ such that
\begin{equation}\label{mmm}
|{}^{g_0}\nabla^{\ell}(g-g_0)|_{g_0}+|{}^{g_0}\nabla^{\bar{\ell}}(k-k_0)|_{g_0}
+|E-E_0|_{g_0}+|B-B_0|_{g_0}=O\left(s^{-1}\right)\text{ }\text{ }\text{ as }\text{ }\text{ }s\rightarrow\infty
\end{equation}
for $\ell=0,1,2$ and $\bar{\ell}=0,1$, where
$g_0=\xi_0(ds^2+\sigma_0)$ in which $\xi_0>0$, $\sigma_0$ are a function and metric on $S^2$, and the background data are invariant under radial translations
\begin{equation}
\mathfrak{L}_{\partial_s}g_0=\mathfrak{L}_{\partial_s}k_0=
\mathfrak{L}_{\partial_s}E_0=\mathfrak{L}_{\partial_s}B_0=0.
\end{equation}
The presence of $\xi_0$ makes the background metric conformally cylindrical, and this generalized notion of asymptotically cylindrical ends is used to accommodate examples such as the extreme Kerr-Newman geometry.

Consider normal variations, inside $\Omega$, of a future apparent horizon $(\Sigma,\sigma)$ with speed $\varphi\in C^{\infty}(\Sigma)$. Then the variation of the future null expansion (see \cite{AnderssonMarsSimon}) is given by
\begin{equation}
D\theta_{+}[\varphi]=\mathcal{L}\varphi:=-\Delta_{\sigma}\varphi+2\langle V,\nabla \varphi\rangle
+(W+\operatorname{div}_{\sigma}V-|V|_{\sigma}^2)\varphi,
\end{equation}
where
\begin{equation}
W=K_{\sigma}-8\pi(\mu+J(\nu))-\frac{1}{2}|II|_{\sigma}^2,\quad\quad\quad V=p(\nu)^T,
\end{equation}
with $K_{\sigma}$ denoting Gauss curvature and $II$ representing the null second fundamental form associated with $l_+$. Although $\mathcal{L}$ is not necessarily self adjoint, the principal eigenvalue $\lambda_1$ is real and simple. The future apparent horizon $\Sigma$ is referred to as \textit{stable} if $\lambda_1\geq 0$. A similar statement holds for past apparent horizons. The limiting cross-section of an asymptotically cylindrical end will be referred to as stable if the principal eigenvalue of the limiting stability operator is nonnegative.

\begin{proposition}\label{prop1}
Let $(\Omega,g,k,E,B)$ be an initial data set for the Einstein-Maxwell equations
satisfying the charged dominant energy condition $\mu_{EM}\geq |J_{EM}|_g$.
\begin{enumerate}
\item If $\Sigma\subset\Omega$ is a single component axisymmetric stable apparent horizon, or

\item $\Sigma$ is the axisymmetric stable limiting cross-section of an asymptotically cylindrical end in $\Omega$,
\end{enumerate}
then
\begin{equation}\label{areaamcharge}
|\Sigma|\geq 4\pi\sqrt{Q^4+4\mathcal{J}^2},
\end{equation}
and equality is achieved if and only if $(\Sigma,\sigma,k(\nu)^T,E,B)$ arises from an extreme Kerr-Newman horizon. Furthermore, the same conclusion holds without the contribution from $\mathcal{J}$ if the assumption of axisymmetry is dropped, and the same conclusion holds without the contribution from $Q$ if the energy condition is changed to $\mu\geq|J|_g$.
\end{proposition}

\begin{proof}
Part (1) is established in \cite{BrydenKhuri,Dain2012,ClementJaramilloReiris,ClementReirisSimon}. Part (2) follows from a straightforward generalization of the proof in (1). More precisely,
by taking limits the functions $k(\eta,\nu)$, $E(\nu)$, $B(\nu)$, as well as the induced metric $\sigma$ are well-defined on the limiting cross-section $\Sigma$. From these, a map $\Psi:S^2\rightarrow\mathbb{H}^2_{\mathbb{C}}$ may be constructed
with the complex hyperbolic plane as target. The stability assumption implies (see \cite[Section 3]{BrydenKhuri}) that
\begin{equation}\label{jjjj}
|\Sigma|\geq 4\pi e^{\frac{\mathcal{I}(\Psi)-c}{4\pi}}
\end{equation}
for a universal constant $c$, where $\mathcal{I}(\Psi)$ denotes the renormalized harmonic energy. The functional $\mathcal{I}(\Psi)$ is then minimized among all maps with fixed angular momentum and charge, and it is shown that the unique minimizer is the map $\Psi_{kn}$ arising from the extreme Kerr-Newman black hole with the given angular momentum and charge \cite[Section 4]{BrydenKhuri}. The desired result now follows from \eqref{jjjj} and a computation to determine the value of $\mathcal{I}(\Psi_{kn})$. The case of equality is treated in the same way as in \cite{BrydenKhuri}, and similar considerations hold for the situation when the assumption of axisymmetry is dropped and the contribution from $\mathcal{J}$ is removed from \eqref{areaamcharge}.
\end{proof}

\begin{proof}[Proof of Theorem \ref{thmBYchargein}]
Consider the composite manifold $(\mathbf{M},\mathbf{g})=(\Omega\cup M_+, g\cup g_+)$ described in Section \ref{sec3}, in which the Bartnik-Shi-Tam extension is glued onto the given initial data. On each asymptotically cylindrical end remove the tail piece for $s>s_0$, where $s_0$ is a large parameter, and denote the result $(\mathbf{M}_{s_0},\mathbf{g})$. Let $\chi(s)$ be a smooth cut-off function with  $\chi=1$ for $s\leq s_0-2$ and $\chi=0$ for $s\geq s_0-1$, and consider the metric
$\mathbf{g}_{s_0}=\chi\mathbf{g}+(1-\chi)g_0$ where $g_0$ is the cylindrical model metric as in \eqref{mmm}. Then $(\mathbf{M}_{s_0},\mathbf{g}_{s_0})$ is exactly cylindrical near the parts of its boundary arising from the asymptotically cylindrical ends, $\partial_{cyl}\mathbf{M}_{s_0}$. Now double this manifold across its minimal surface boundary, and note that the result is smooth across $\partial_{cyl}\mathbf{M}_{s_0}$. Although the resulting doubled manifold is not smooth at $\Sigma=\partial\Omega$ and $\partial\mathbf{M}_{s_0}\setminus\partial_{cyl}\mathbf{M}_{s_0}$, the induced metrics and mean curvatures from both sides agree so that the scalar curvature is nonnegative in the distributional sense. This allows for a smoothing of the metric from \cite{Miao}, denoted
$\mathbf{g}_{s_0}^{\delta}$, followed by a conformal deformation $\tilde{\mathbf{g}}_{s_0}^{\delta}=(u_{s_0}^{\delta})^4 \mathbf{g}_{s_0}^{\delta}$ to nonnegative scalar curvature on the doubled manifold which preserves the asymptotically flat ends. We note that although $\mathbf{g}_{s_0}^{\delta}$ does not necessarily have nonnegative scalar curvature, the negative part of its scalar curvature can be made arbitrarily small in the $L^{3/2}$-norm for sufficiently large $s_0$ and sufficiently small $\delta$, and this is adequate for the existence of the conformal factor. Furthermore,
it can be shown from \cite{Jaracz,Miao} that
\begin{equation}\label{poiu}
u_{s_0}^{\delta}\rightarrow 1,\quad\quad
\tilde{\mathbf{m}}_{s_0}^{\delta}\rightarrow\mathbf{m},\quad\text{ as }\quad
(s_0,\delta)\rightarrow(\infty,0),
\end{equation}
where $\tilde{\mathbf{m}}_{s_0}^{\delta}$ and $\mathbf{m}$ are the ADM masses of $\tilde{\mathbf{g}}_{s_0}^{\delta}$ and $\mathbf{g}$.
Moreover this construction can be carried out so that there is a reflection symmetry across the doubling surfaces, and therefore these are minimal surfaces (in fact totally geodesic) with respect to $\tilde{\mathbf{g}}_{s_0}^{\delta}$. Observe that now $(\mathbf{M}_{s_0},\tilde{\mathbf{g}}_{s_0}^{\delta})$ is asymptotically flat with nonnegative scalar curvature and minimal surface boundary. Let $\tilde{\Sigma}_{h}(s_0,\delta)\subset \mathbf{M}_{s_0}$ denote the outermost minimal surface, then the Penrose inequality \cite{Bray} yields
\begin{equation}
\tilde{\mathbf{m}}_{s_0}^{\delta}\geq
\sqrt{\frac{|\tilde{\Sigma}_{h}(s_0,\delta)|_{\tilde{\mathbf{g}}_{s_0}^{\delta}}}{16\pi}}.
\end{equation}
In light of \eqref{poiu}, and the fact that $\mathbf{g}_{s_0}^{\delta}$ is $C^0$ close to $\mathbf{g}$, there are positive constants $c(s_0,\delta)\rightarrow 0$ as $(s_0,\delta)\rightarrow (\infty,0)$ such that
\begin{equation}
|\tilde{\Sigma}_{h}(s_0,\delta)|_{\tilde{\mathbf{g}}_{s_0}^{\delta}}
\geq\left(1-c(s_0,\delta)\right)|\tilde{\Sigma}_{h}(s_0,\delta)|_{\mathbf{g}}
\geq\left(1-c(s_0,\delta)\right)|\Sigma_{*}|_g.
\end{equation}
In the last inequality we have used the outerminimizing property of $\Sigma_*$ in $(\mathbf{M},\mathbf{g})$, which follows from the assumed outerminimizing nature of $\Sigma_*$ in $(\Omega,g)$ together with the positive mean curvature foliation of $M_+$. Since the masses converge we then have
\begin{equation}\label{qw}
m_{BY}(\Sigma)\geq\mathbf{m}\geq\sqrt{\frac{|\Sigma_*|_g}{16\pi}}.
\end{equation}
This may be considered as a generalization of the Penrose inequality with corners,
established in \cite{McCormickMiao,Miao 2009}, to the setting that includes asymptotically cylindrical ends. We may now apply Proposition \ref{prop1} to each component of $\Sigma_*$ to obtain the desired inequality \eqref{BY1in}.

Consider now the setting of \eqref{byq2}, where the generalized boundary $\Sigma_*$ has a single component. It will be assumed that an asymptotically cylindrical end is present, that is $\Sigma_*=\Sigma_{\infty}$, since the proof when $\Sigma_*=\Sigma_h$ is similar and has fewer steps. Observe that the outermost minimal surface $\Sigma_{h}(s_0)\subset(\mathbf{M}_{s_0},\mathbf{g}_{s_0})$ must lie far down the cylindrical end since there are no compact minimal surfaces in $\mathbf{M}$. By \cite[Lemma 4.1]{HuiskenIlmanen} the region outside this surface is diffeomorphic to $\mathbb{R}^3\setminus Ball$. In particular $\Omega$ is simply connected.

We now produce a weak inverse mean curvature flow (IMCF) emanating from the asymptotically cylindrical end. Let $\Sigma_{h}(s_i,\delta)$ denote the
outermost minimal surface in $(\mathbf{M}_{s_i},\mathbf{g}_{s_i}^{\delta})$ where $s_i\rightarrow \infty$,
and let $w_{s_i}^{\delta}$ be the locally Lipschitz level set function defining the IMCF starting at $\Sigma_{h}(s_i,\delta)$. Theorem 3.1 of \cite{HuiskenIlmanen} provides a uniform (independent of $s_i$) $C^1$ bound for $w_{s_i}^{\delta}$. Due to the area outerminimizing property of the level sets, and the fact that the smallest level is approximately $|\Sigma_*|$, it follows that on $\Omega$ the exponential of these functions remain bounded within the interval $(1-\delta,|\Sigma|/|\Sigma_*|+\delta)$ for $i$ sufficiently large. This sequence is then uniformly bounded and equicontinuous on compact subsets, and there is then a subsequence converging locally uniformly to $w^{\delta}$, which is a solution of the weak IMCF \cite[Theorem 2.1]{HuiskenIlmanen} on $(\mathbf{M},\mathbf{g}^{\delta})$.
Similar arguments yield subsequential convergence $w^{\delta}\rightarrow w$ on $\Omega$, where $w$ is a weak solution of IMCF. By subtracting a constant if necessary we may assume that $\inf_{\Omega}w=|\Sigma_*|$.

Let $\{S_t\}_{t=0}^{t_0}$ denote the leaves of the flow ($t$-level sets of $w$), with $S_{t_0}$ denoting the leaf with largest area in $\Omega$. Note that each leaf is connected since $\Omega$ is simply connected \cite[Lemma 4.2]{HuiskenIlmanen}. For any $t_*<t_0$ the leaf $S_{t_*}$ is outerminimizing in $\Omega$. Thus
\cite[Theorem 3.1]{ShiTamHorizons} implies that
\begin{equation}\label{888}
m_{BY}(\Sigma)\geq m_{H}(S_{t_*}):=\sqrt{\frac{|S_{t_*}|}{16\pi}}\left(1-
\frac{1}{16\pi}\int_{S_{t_*}}H^2 dA_t\right),
\end{equation}
where $m_{H}$ is Hawking mass. Furthermore, monotonicity of the Hawking mass produces
\begin{align}\label{8889}
\begin{split}
m_{H}(S_{t_*})-\sqrt{\frac{|\Sigma_*|}{16\pi}}\geq&
\frac{1}{(16\pi)^{3/2}}\int_{0}^{t_*}\int_{S_t}
\sqrt{|S_t|}R dA_t dt\\
\geq &\frac{2}{(16\pi)^{3/2}}\int_{0}^{t_*}\sqrt{|S_t|}\int_{S_t}|E|^2 dA_t dt\\
\geq& \frac{2}{(16\pi)^{3/2}}\int_{0}^{t_*}\frac{1}{\sqrt{|S_t|}}
\left(\int_{S_t}E\cdot\nu dA_t\right)^2 dt\\
=&\frac{2(4\pi)^{2}Q^2}{(16\pi)^{3/2}}\int_{0}^{t_*}
\frac{e^{-t/2}}{\sqrt{|\Sigma_*|}}dt\\
=&\sqrt{\frac{\pi}{|\Sigma_*|}}Q^2\left(1-e^{-t_{*}/2}\right),
\end{split}
\end{align}
where we have used $|S_t|=|\Sigma_*|e^t$ and are assuming without loss of generality that the magnetic charge vanishes so that $Q^2=Q_e^2$; this may be accomplished by performing a duality rotation of the electromagnetic field.
It now follows from Proposition \ref{prop1} that
\begin{equation}
m_{BY}(\Sigma)\geq\frac{|Q|}{2}+\sqrt{\frac{\pi}{|\Sigma_*|}}Q^2 -\sqrt{\frac{\pi}{|S_{t_*}|}}Q^2.
\end{equation}
The desired inequality \eqref{byq2} is achieved by letting $t_*\rightarrow t_0$.
\end{proof}

\begin{proof}[Proof of Theorem \ref{thmBYAMin}]
Inequality \eqref{BY2in} follows from \eqref{qw} and an application of Proposition \ref{prop1}. Consider now the setting of \eqref{byam2}, in which the generalized boundary $\Sigma_*$ has a single component. Here we may apply \eqref{888} and a modified version of \eqref{8889} to obtain the desired result. Namely, let $\eta$ be the rotational Killing field then the dominant energy condition and H\"{o}lder's inequality imply
\begin{align}
\begin{split}
m_{H}(S_{t_*})-\sqrt{\frac{|\Sigma_*|}{16\pi}}\geq&
\frac{1}{(16\pi)^{3/2}}\int_{0}^{t_*}\int_{S_t}
\sqrt{|S_t|}R dA_t dt\\
\geq &\frac{1}{(16\pi)^{3/2}}\int_{0}^{t_*}\sqrt{|S_t|}\int_{S_t}|k|^2 dA_t dt\\
\geq &\frac{1}{(16\pi)^{3/2}}\int_{0}^{t_*}\sqrt{|S_t|}\int_{S_t}|k(|\eta|^{-1}\eta,\nu)|^2 dA_t dt\\
\geq& \frac{1}{(16\pi)^{3/2}}\int_{0}^{t_*}\frac{\sqrt{|S_t|}}{\int_{S_t}|\eta|^2}
\left(\int_{S_t}k(\eta,\nu)dA_t\right)^2 dt\\
\geq&\frac{\sqrt{\pi}\mathcal{J}^2}
{\max_t \tfrac{1}{|S_t|}\int_{S_t}|\eta|^2}\int_{0}^{t_*}\frac{e^{-t/2}}{\sqrt{|\Sigma_*|}}dt\\
\geq&\frac{(2\pi)^2\mathcal{J}^2}{\mathcal{C}^2}\left(\sqrt{\frac{4\pi}{|\Sigma_*|}}
-\sqrt{\frac{4\pi}{|S_{t_*}|}}\right)
=\frac{(2\pi\alpha)^2}{\mathcal{C}^2}\sqrt{\frac{4\pi}{|\Sigma_*|}}\mathcal{J}^2,
\end{split}
\end{align}
where $\mathcal{C}=2\pi\max_{\Omega}|\eta|$ is the largest circumference of all $\eta$-orbits in $\Omega$. Proposition \ref{prop1} can now be used together with $t_*\rightarrow t_0$ to obtain inequality \eqref{byam2}.
\end{proof}

\begin{proof}[Proof of Theorem \ref{combinedamch}]
Inequality \eqref{combined1} follows from \eqref{qw} and an application of Proposition \ref{prop1}. Consider now the case when $\Sigma_*$ has a single component. By combining the proofs of Theorems \ref{thmBYchargein} and \ref{thmBYAMin} we obtain
\begin{equation}
m_{BY}(\Sigma)\geq\sqrt{\frac{|\Sigma_*|}{16\pi}}
+\alpha^2 \sqrt{\frac{\pi}{|\Sigma_*|}} Q^2
+\frac{(2\pi\alpha)^2}{\mathcal{C}^2}\sqrt{\frac{4\pi}{|\Sigma_*|}}\mathcal{J}^2.
\end{equation}
Therefore
\begin{align}
\begin{split}
m_{BY}(\Sigma)^2\geq & \left(\sqrt{\frac{|\Sigma_*|}{16\pi}}
+\alpha^2 \sqrt{\frac{\pi}{|\Sigma_*|}} Q^2\right) m_{BY}(\Sigma)\\
\geq&\left(\sqrt{\frac{|\Sigma_*|}{16\pi}}
+\alpha^2 \sqrt{\frac{\pi}{|\Sigma_*|}} Q^2\right)^2
+\sqrt{\frac{|\Sigma_*|}{16\pi}}\frac{(2\pi\alpha)^2}{\mathcal{C}^2}
\sqrt{\frac{4\pi}{|\Sigma_*|}}\mathcal{J}^2\\
=&\left(\sqrt{\frac{|\Sigma_*|}{16\pi}}
+\alpha^2 \sqrt{\frac{\pi}{|\Sigma_*|}} Q^2\right)^2
+\frac{\alpha^2}{2}\frac{\mathcal{J}^2}{\mathcal{R}^2_c},
\end{split}
\end{align}
where $\mathcal{R}_c=\mathcal{C}/2\pi$ is the circumference radius. The desired inequality \eqref{combined2} now follows.
\end{proof}

\section{Gluing for the Jang Surface and Wang-Yau Mass Positivity in Higher Dimensions}\label{sec5}\setcounter{equation}{0}
\setcounter{section}{5}

Let $(\Omega,g,k,E)$ be a compact spacelike hypersurface with induced electric field. The magnetic field will be ignored here since a duality rotation may always be applied to ensure that the total charge $Q$ agrees with the electric charge $Q_e$ up to sign. If the boundary $\Sigma=\partial\Omega$ is untrapped so that $\theta_{\pm}(\Sigma)>0$, local barriers can be constructed to obtain boundary gradient estimates for the boundary value problem \eqref{Jang}, and this leads to a solution $f$ for any given regular Dirichlet data $\tau$. The Jang deformation of the initial data set is then defined to be $(\bar{\Omega},\bar{g},\bar{E})$ where
\begin{equation}
\bar{g}_{ij}=g_{ij}+f_i f_j,\quad\quad\quad\quad
\bar{E}_i:=\frac{E_i +f_i f^j E_{j}}{\sqrt{1+|\nabla f|^2_g}}.
\end{equation}
Note that if the Jang graph blows-up at apparent horizons then the topology of $\bar{\Omega}$ may differ from that of $\Omega$.
It is shown in \cite{DisconziKhuri} that if $E$ is divergence free then so is $\bar{E}$ and an inequality holds between the energy densities
\begin{equation}\label{barEproperties}
|E|_g\geq |\bar{E}|_{\bar{g}},\quad\quad\quad \text{div}_{\bar{g}}\bar{E}=0.
\end{equation}
Moreover if regular level sets of $f$ are homologous to $\Sigma$ then the total charge of the deformation is equal to that of the original
\begin{equation}\label{kjkl}
\bar{Q}=Q.
\end{equation}
To see this observe that on a level set having unit normals $\nu$, $\bar{\nu}$ with respect to $g$, $\bar{g}$ we have
\begin{equation}
\bar{\nu}^i=\nu^i \sqrt{1+|\nabla f|_g^2}-\frac{f^i \nu(f)}{\sqrt{1+|\nabla f|_g^2}}\quad\quad \Rightarrow \quad\quad \bar{E}_i \bar{\nu}^i =E_i \nu^i,
\end{equation}
so that \eqref{kjkl} follows from the divergence theorem. In particular, this equality of total charges occurs when the Jang surface blows-up at the outermost apparent horizon. Furthermore the scalar curvature formula
\eqref{Jangscalar}, the energy density relation of \eqref{barEproperties}, and the energy condition $\mu_{EM}\geq |J|_{g}$ imply
\begin{equation}\label{inequality11}
\bar{R}-2|X|_{\bar{g}}^2 +2\text{div}_{\bar{g}}X
=2(\mu-J(w))+|h-k|_{\bar{g}}^2\geq 2|\bar{E}|_{\bar{g}}^2+|\bar{k}|_{\bar{g}}^2,
\end{equation}
where $\bar{k}=h-k$. If only the dominant energy condition holds $\mu\geq |J|_g$, then the same inequality \eqref{inequality11} holds without the electric field contribution on the right-hand side.

Let $(\bar{M}_+,\bar{g}_+,\bar{k}_+=0, X_+=0,\bar{E}_+=0)$ be a Bartnik-Shi-Tam extension of the Jang initial data, with zero scalar curvature and boundary $\partial \bar{M}_+=\Sigma$ having mean curvature
\begin{equation}
H_+=\bar{H}-X(\bar{\nu}),
\end{equation}
where $\bar{H}$ is mean curvature of $\Sigma$ with respect to $\bar{g}$ and the unit normal $\bar{\nu}$ is pointing out of $\bar{\Omega}$. By definition, if $\tau$ is admissible with $\Sigma$ then $H_+$ is positive so that the existence of the extension is not inhibited; if $\tau=0$ then positivity of $H_+$ follows from the untrapped condition \cite{Yau}. We may now construct composite asymptotically flat data
\begin{equation}
(\bar{\mathbf{M}},\bar{\mathbf{g}},\bar{\mathbf{k}},\mathbf{X},\bar{\mathbf{E}})
=(\bar{\Omega}\cup \bar{M}_+, \bar{g}\cup \bar{g}_+, \bar{k}\cup \bar{k}_+, X\cup X_+,\bar{E}\cup \bar{E}_+),
\end{equation}
which exhibits a corner at the hypersurface $\Sigma$. A key observation is that this corner may be smoothed while preserving the stability type inequality \eqref{inequality11} in a weak sense.

\begin{lemma}\label{gluing}
There exists a smooth deformation $(\bar{\mathbf{M}},\bar{\mathbf{g}}_{\delta},\bar{\mathbf{k}}_{\delta},
\mathbf{X}_{\delta},\bar{\mathbf{E}}_{\delta})$ which differs from the original data $(\bar{\mathbf{M}},\bar{\mathbf{g}},\bar{\mathbf{k}},\mathbf{X},\bar{\mathbf{E}})$ only on a $\delta$-tubular neighborhood of $\Sigma$, and satisfies
\begin{equation}\label{6yj}
\bar{\mathbf{R}}_{\delta}-2|\mathbf{X}_{\delta}|_{\bar{\mathbf{g}}_{\delta}}^2
	+2\operatorname{div}_{\bar{\mathbf{g}}_{\delta}}\mathbf{X}_{\delta}\geq
2|\bar{\mathbf{E}}_{\delta}|_{\bar{\mathbf{g}}_{\delta}}^2
+|\bar{\mathbf{k}}_{\delta}|_{\bar{\mathbf{g}}_{\delta}}^2
+O(1)\quad\quad\text{ as }\quad\quad \delta\rightarrow 0,
\end{equation}
where $\bar{\mathbf{R}}_{\delta}$ is the scalar curvature of $\bar{\mathbf{g}}_{\delta}$.
\end{lemma}

\begin{proof}
For the metric we select the same deformation as in \cite[Section 3]{Miao}.
Since a similar mollification will be used for other quantities, we record here the basic set up. Fix $\varepsilon>0$ and write the metric near the corner $(-2\varepsilon,2\varepsilon)\times \Sigma\subset \bar{\mathbf{M}}$ in Gaussian coordinates as
\begin{equation}
\bar{\mathbf{g}}=dt^2+\gamma_{ij}(t,x)dx^i dx^j,
\end{equation}
where $\gamma(t)$ is a continuous curve of metrics on $\Sigma$.
Let $0<\delta\ll\varepsilon$ and $\phi\in C^{\infty}_{c}((-1,1))$ be the standard mollifier with $0\leq \phi(s)\leq 1$ and $\int_{-1}^1\phi ds=1$. Moreover, set $\sigma_{\delta}(t)=\delta^2\sigma(t/\delta)$ where $\sigma\in C^{\infty}_{c}((-\tfrac{1}{2},\tfrac{1}{2}))$ satisfies
\begin{equation}
\sigma(t)=\begin{cases}
\sigma(t)= \frac{1}{100} & |t|\leq \frac{1}{4}\\
0<\sigma(t)\leq \frac{1}{100} & \frac{1}{4}<|t|<\frac{1}{2}
\end{cases}.
\end{equation}
For $t\in (-\varepsilon,\varepsilon)$ the smoothed curve of metrics on $\Sigma$ is given by
\begin{equation}
\gamma_{\delta}(t,x)=\int_{\mathbb{R}}\gamma(t-\sigma_{\delta}(t)s,x)\phi(s) ds=\begin{cases}	
\int_{\mathbb{R}}\gamma(s,x)\frac{1}{\sigma_{\delta}(t)}
\phi(\frac{t-s}{\sigma_{\delta}(t)}) ds & \sigma_{\delta}(t)>0\\
\gamma(t) & \sigma_{\delta}(t)=0
\end{cases},
\end{equation}
so that
\begin{equation}\label{deformmetric}
\bar{\mathbf{g}}_{\delta}=\begin{cases}
dt^2+\gamma_\delta(t,x) & (t,x)\in (-\varepsilon,\varepsilon)\times\Sigma\\
\bar{\mathbf{g}} & (t,x)\notin (-\varepsilon,\varepsilon)\times\Sigma\\
\end{cases}
\end{equation}
is the desired smoothing of $\bar{\mathbf{g}}$. It is shown in \cite{Miao} that the scalar curvature of the deformation satisfies
\begin{equation}
\bar{\mathbf{R}}_{\delta}=O(1),\qquad (t,x)\in\left\{\frac{\delta^2}{100}<|t|\leq\frac{\delta}{2}\right\}\times \Sigma,
\end{equation}
\begin{equation}
\bar{\mathbf{R}}_{\delta}=O(1)+\left(H_-(x)-H_{+}(x)\right)\frac{100}{\delta^2}
\phi\left(\frac{100 t}{\delta^2}\right),\qquad (t,x)\in\left[-\frac{\delta^2}{50},\frac{\delta^2}{50}\right]\times \Sigma,
\end{equation}
where $H_-$, $H_+$ are the mean curvatures of $\Sigma$ with respect to $\bar{g}$ and $\bar{g}_+$.

Next, near the corner surface $\Sigma$ write
\begin{equation}
\mathbf{X}=\mathbf{X}_t dt+\mathbf{X}_i dx^i,
\end{equation}
and denote the smoothed 1-form by
\begin{equation}
\mathbf{X}_{\delta}= \mathbf{X}_{\delta t} dt + \mathbf{X}_{\delta i} dx^i.
\end{equation}
The $\mathbf{X}_{\delta i}$ are any smoothing of the tangential components which keeps tangential derivatives bounded and agrees with $\mathbf{X}_i$ outside the $\delta$-tubular neighborhood of $\Sigma$, and the smoothed normal component is defined by
\begin{equation}
\mathbf{X}_{\delta t}=\int_{\mathbb{R}} \mathbf{X}_t(t-2\sigma_{\delta}(t)s)
\phi(\varsigma_{\delta}(t)s) ds,
\end{equation}
where $\varsigma_{\delta}(t)=1$ for $|t|>\varepsilon$ and $\varsigma_{\delta}(t)=2$ for $|t|<\frac{\delta}{4}$ with $|\varsigma'|\leq 2\varepsilon^{-1}$. Then for $|t|>\frac{\delta^2}{50}$ we have
\begin{equation}
\partial_t \mathbf{X}_{\delta t}=\int_{\mathbb{R}} \left[ \mathbf{X}_{t}'(t-2\sigma_{\delta}(t)s)\left(1-2\sigma_{\delta}'(t)s\right)
\phi(\varsigma_{\delta}(t)s)+\mathbf{X}_t(t-2\sigma_{\delta}(t)s)
\phi'(\varsigma_{\delta}(t)s)\varsigma_{\delta}'(t)\right] ds,
\end{equation}
and for $|t|<\frac{\delta}{4}$ since $\sigma_{\delta}(t)=\frac{\delta^2}{100}$ and $\varsigma_{\delta}(t)=2$ it follows that
\begin{align}
\begin{split}
\partial_t \mathbf{X}_{\delta t}&=\partial_t \int_{\mathbb{R}} \mathbf{X}_t \left(t-\tfrac{\delta^2}{50}s\right)\phi(2s) ds\\
=&\partial_t \int_{\mathbb{R}} \mathbf{X}_t (s)\frac{50}{\delta^2}\phi\left(\tfrac{100(t-s)}{\delta^2}\right) ds\\
=&-\int_{\mathbb{R}} \mathbf{X}_t(s)\frac{50}{\delta^2}\partial_s\phi\left(\frac{100(t-s)}{\delta^2}\right) ds\\
=&\left(\mathbf{X}_{t+}(x)-\mathbf{X}_{t-}(x)\right)\frac{50}{\delta^2}
\phi\left(\frac{100t}{\delta^2}\right)\\
&+\int_{-\infty}^0 \partial_s \mathbf{X}_{t}(s)\frac{50}{\delta^2}\phi\left(\tfrac{100(t-s)}{\delta^2}\right) ds+\int_{0}^{\infty} \partial_s \mathbf{X}_{t}(s)\frac{50}{\delta^2}
\phi\left(\tfrac{100(t-s)}{\delta^2}\right) ds,
\end{split}
\end{align}
where $\mathbf{X}_{t \pm}=\lim_{t\rightarrow 0^{\pm}}\mathbf{X}_t$.
In particular we have
\begin{equation}\label{..}
\partial_t \mathbf{X}_{\delta t}=O(1),\qquad (t,x)\in\left\{\frac{\delta^2}{50}<|t|\leq\frac{\delta}{2}\right\}\times \Sigma,
\end{equation}
and
\begin{equation}\label{...}
\partial_t \mathbf{X}_{\delta t}=
\left(\mathbf{X}_{t+}(x)-\mathbf{X}_{t-}(x)\right)\frac{50}{\delta^2}
\phi\left(\frac{100t}{\delta^2}\right)
+O(1),\qquad (t,x)\in\left[-\frac{\delta^2}{50},\frac{\delta^2}{50}\right]\times \Sigma.
\end{equation}

Consider now the divergence of the smoothed 1-form
\begin{align}\label{p}
\begin{split}
\text{div}_{\bar{\mathbf{g}}_{\delta}}\mathbf{X}_{\delta}
=&\frac{1}{\sqrt{\det \bar{\mathbf{g}}_{\delta}}}
\partial_a \left(\sqrt{\det\bar{\mathbf{g}}_{\delta}}\bar{\mathbf{g}}_{\delta}^{ab}
\mathbf{X}_{\delta b}\right)\\
=&\frac{1}{\sqrt{\det \gamma_{\delta}}}\left[
\partial_t\left(\sqrt{\det\gamma_{\delta}} \mathbf{X}_{\delta t}\right)
+
\partial_i \left(\sqrt{\det\gamma_{\delta}}\gamma_{\delta}^{ij}
\mathbf{X}_{\delta j}\right)\right]\\
=&\partial_t \mathbf{X}_{\delta t}+\frac{1}{2}\mathbf{X}_{\delta t}\partial_t \log\det\gamma_{\delta}
+\frac{1}{\sqrt{\det \gamma_{\delta}}}
\partial_i \left(\sqrt{\det\gamma_{\delta}}\gamma_{\delta}^{ij}
\mathbf{X}_{\delta j}\right).
\end{split}
\end{align}
Together with \eqref{..} and \eqref{...} this yields
\begin{equation}
\text{div}_{\bar{\mathbf{g}}_{\delta}}\mathbf{X}_{\delta}=O(1),\qquad (t,x)\in\left\{\frac{\delta^2}{50}<|t|\leq\frac{\delta}{2}\right\}\times \Sigma,
\end{equation}
and
\begin{equation}
\text{div}_{\bar{\mathbf{g}}_{\delta}}\mathbf{X}_{\delta}
=\left(\mathbf{X}_{t+}(x)-\mathbf{X}_{t-}(x)\right)\frac{50}{\delta^2}
\phi\left(\frac{100t}{\delta^2}\right)
+O(1),\qquad (t,x)\in\left[-\frac{\delta^2}{50},\frac{\delta^2}{50}\right]\times \Sigma.
\end{equation}
Therefore
\begin{equation}
\bar{\mathbf{R}}_{\delta}-2|\mathbf{X}_{\delta}|_{\bar{\mathbf{g}}_{\delta}}^2
+2\text{div}_{\bar{\mathbf{g}}_{\delta}}\mathbf{X}_{\delta}=O(1),\qquad (t,x)\in\left\{\frac{\delta^2}{50}<|t|\leq\frac{\delta}{2}\right\}\times \Sigma,
\end{equation}
and
\begin{align}
\begin{split}
\bar{\mathbf{R}}_{\delta}-2|\mathbf{X}_{\delta}|_{\bar{\mathbf{g}}_{\delta}}^2
+2\text{div}_{\bar{\mathbf{g}}_{\delta}}\mathbf{X}_{\delta}=&
\left(H_-(x)-\mathbf{X}_{t-}(x)+\mathbf{X}_{t+}(x)-H_{+}(x)\right)\frac{100}{\delta^2}
\phi\left(\frac{100 t}{\delta^2}\right)+O(1),\\
&(t,x)\in[-\frac{\delta^2}{50},\frac{\delta^2}{50}]\times \Sigma.
\end{split}
\end{align}
Since $H_{+}=\bar{H}-X(\bar{\nu})$, $H_-=\bar{H}$, $\mathbf{X}_{t+}=0$, and $\mathbf{X}_{t-}=X(\bar{\nu})$ we find that
\begin{equation}
\bar{\mathbf{R}}_{\delta}-2|\mathbf{X}_{\delta}|_{\bar{\mathbf{g}}_{\delta}}^2
+2\text{div}_{\bar{\mathbf{g}}_{\delta}}\mathbf{X}_{\delta}=O(1),\qquad (t,x)\in\left\{|t|\leq\frac{\delta}{2}\right\}\times \Sigma.
\end{equation}

The smoothings $\bar{\mathbf{E}}_{\delta}$ and $\bar{\mathbf{k}}_{\delta}$ may be achieved in a similar manner. Since derivatives of these quantities do no appear in \eqref{6yj}, we only require that they remain uniformly bounded and agree with $\bar{\mathbf{E}}$ and $\bar{\mathbf{k}}$ for $|t|>\frac{\delta}{2}$. The desired result \eqref{6yj} now follows from \eqref{inequality11} which holds on both sides of $\Sigma$.
\end{proof}

Lemma \ref{gluing} shows that $\bar{\mathbf{R}}_{\delta}-2|\mathbf{X}_{\delta}|_{\bar{\mathbf{g}}_{\delta}}^2
+2\text{div}_{\bar{\mathbf{g}}_{\delta}}\mathbf{X}_{\delta}$ is nonnegative in a weak sense as long as the dominant energy condition holds. More precisely, this quantity is nonnegative except on a small set, where it remains uniformly bounded as $\delta\rightarrow 0$. This ensures that the smoothed Jang metric $\bar{\mathbf{g}}_{\delta}$ may be conformally deformed to zero scalar curvature as in the Schoen-Yau proof of the positive mass theorem \cite{SchoenYauII}. These arguments, as well as Lemma \ref{gluing}, extend to higher dimensions and lead to a new proof of nonnegativity for the Wang-Yau mass that does not require spinors. The spin assumption needed for previous versions of the result can then be removed.

In order to state the higher dimensional result, a revised definition of admissibility will be given. Consider a codimension two closed spacelike submanifold $(\Sigma,\sigma)$ of a $d+1$-dimensional spacetime $N^{d,1}$. We will say that the surface $\Sigma$ and time function $\tau$ are \textit{generally admissible} if (1) the metric $\hat{\sigma}=\sigma+d\tau^2$ on $\Sigma$
has positive scalar curvature and is isometric to a mean convex star-shaped hypersurface in $\mathbb{R}^d$, (2)
$\Sigma$ arises as the untrapped boundary of a spacelike hypersurface $(\Omega,g,k)\hookrightarrow N^{d,1}$, and (3) the generalized mean curvature is positive $\mathfrak{H}(e_{d}',e_{d+1}')>0$ for the normal bundle frame $\{e_{d}',e_{d+1}'\}$ determined by the solution of Jang's equation. Note that condition (1) guarantees an isometric embedding into Minkowski space $\iota:\Sigma\hookrightarrow\mathbb{R}^{d,1}$, where $(\hat{\Sigma},\hat{\sigma})$ is the projection of $\iota(\Sigma)$ onto a constant time slice. From this one may define the Wang-Yau energy as in \eqref{wangyauenergy}.

\begin{theorem}
Let $(\Sigma,\sigma)$ be a codimension two closed spacelike submanifold of a spacetime $N^{d,1}$, $3\leq d\leq 7$, satisfying the dominant energy condition.
If $\tau$ is a time function that is generally admissible with $\Sigma$ and $\iota:\Sigma\hookrightarrow\mathbb{R}^{d,1}$ is an associated isometric embedding, then the Wang-Yau energy is nonnegative $E_{WY}(\Sigma,\iota,\tau)\geq 0$. Equality occurs if and only if the spacelike hypersurface $(\Omega,g,k)$, whose boundary is $(\Sigma,\sigma)$, arises from an embedding into Minkowski space. In particular, the Wang-Yau mass is nonnegative $m_{WY}(\Sigma)\geq 0$.
\end{theorem}

\begin{proof}
As in the beginning of this section let $(\bar{\Omega},\bar{g})$ be the Jang deformation of the initial data $(\Omega,g,k)$, where the solution of Jang's equation satisfies the Dirichlet boundary condition $f=\tau$. The dimensional restriction $3\leq d\leq 7$ is used to obtain locally regular solutions of the Jang equation \cite{Eichmair1}, the theory of which parallels closely that of stable minimal hypersurfaces. Observe that
the induced metric on $\partial\bar{\Omega}$ is then $\hat{\sigma}=\sigma+d\tau^2$. Since $\tau$ is generally admissible, (1) guarantees that there exists a generalized Bartnik-Shi-Tam extension \cite{EichmairMiaoWang} denoted $(\bar{M}_+,\bar{g}_+,X_{+}=0)$ with boundary mean curvature
\begin{equation}
H_+=\bar{H}-X(\bar{\nu})=\frac{\mathfrak{H}(e_{d}',e_{d+1}')}{\sqrt{1+|\nabla\tau|^2}}.
\end{equation}
Note that this quantity is positive by property (3) of the generally admissible condition. The composite manifold $(\bar{\mathbf{M}},\bar{\mathbf{g}},\mathbf{X})$
has a corner at $\Sigma$ which may be smoothed according to Lemma \ref{gluing} to obtain $(\bar{\mathbf{M}},\bar{\mathbf{g}}_{\delta},\mathbf{X}_{\delta})$. The mass $\bar{\mathbf{m}}$ of this smoothed manifold agrees with that of the Bartnik-Shi-Tam extension and thus satisfies (see \cite{EichmairMiaoWang,WangYauPMT1})
\begin{equation}
E_{WY}(\Sigma,\iota,\tau)\geq \frac{1}{8\pi}\int_{\hat{\Sigma}}\left(\hat{H}_0
-\frac{\mathfrak{H}(e_3',e_4')}{\sqrt{1+|\nabla\tau|^2}}\right)dA_{\hat{\sigma}}
\geq\bar{\mathbf{m}},
\end{equation}
where $\hat{H}_0$ is the mean curvature of the projection $(\hat{\Sigma},\hat{\sigma})$ in $\mathbb{R}^d$.

It remains to show that the mass $\bar{\mathbf{m}}$ is nonnegative. The inequality \eqref{6yj} implies that one may follow the steps of the positive mass theorem \cite{SchoenYauII} to derive this conclusion. The generalization of this approach to dimensions $d>3$ encounters some technical difficulties, particularly with the conformal darning of the asymptotically cylindrical ends produced by blow-up of the Jang equation. However these have been resolved by Eichmair in
\cite{Eichmair1} for $3\leq d\leq 7$.  Ultimately one arrives at an asymptotically flat manifold of nonnegative scalar curvature having mass $\tilde{\mathbf{m}}\leq\bar{\mathbf{m}}$. Finally by appealing to the Riemannian version of the positive mass theorem \cite{SchoenYauIV} we find that $\tilde{\mathbf{m}}\geq 0$, from which the desired result follows. The rigidity statement of the theorem is established in the typical manner \cite{Eichmair1,SchoenYauII}.
\end{proof}

\section{Mass Lower Bound from the Conformal Factor}
\label{sec6}\setcounter{equation}{0}
\setcounter{section}{6}

Let $(\Omega,g,k,E)$ be a compact initial data set with boundary $\partial\Omega=\Sigma_h\cup\Sigma$, where $\Sigma_h$ is an outermost apparent horizon in that no other closed apparent horizons are present, and the single component $\Sigma$ is untrapped. There then exists \cite{AnderssonMetzger,Eichmair,HanKhuri} a solution of Jang's equation
\eqref{Jang} which admits asymptotically cylindrical blow-up at $\Sigma_h$ and satisfies $f=\tau$ on $\Sigma$. Consider then the smoothed data $(\bar{\mathbf{M}},\bar{\mathbf{g}}_{\delta},\bar{\mathbf{k}}_{\delta},
\mathbf{X}_{\delta},\bar{\mathbf{E}}_{\delta})$ associated with the Bartnik-Shi-Tam extension of the Jang deformation, as constructed in Lemma \ref{gluing}. We seek a lower bound for its mass $\bar{\mathbf{m}}$, arising from a conformal transformation, that encodes contributions from angular momentum, charge, and horizon area.

In order to impose appropriate boundary conditions for the conformal factor, the various components of $\Sigma_h=\cup_{i=1}^{I}\Sigma_h^i$ will be divided into two groups. Recall that the portion of $(\bar{\mathbf{M}},\bar{\mathbf{g}}_{\delta})$ associated with the Jang surface $(\bar{\Omega},\bar{g})$, may be viewed as a graph $t=f(x)$ lying inside the product manifold $(\mathbb{R}\times\Omega,dt^2+g)$. For each $T>0$ let $\bar{\mathbf{M}}_T$ denote the parts of $\bar{\mathbf{M}}$ consisting of the extension $\bar{M}_+$ outside $\Sigma$ and the portion of $\bar{\Omega}$
lying between the two hyperplanes $t=\pm T$. Due to the asymptotically cylindrical blow-up of the Jang surface the intrinsic geometry of each boundary component $\partial_i\bar{\mathbf{M}}_T$, $i=1,\ldots,I$ approximates that of $\Sigma_h^i$, for $T$ large.
Let $\chi_{T}$ denote the one parameter family of functions on a given boundary component $\Sigma_h^i$ defined as the restriction
of $|X|_{\bar{g}}$ to $\Sigma_h^i$. The parametric estimates for the Jang equation \cite{SchoenYauII} guarantee that the sequence $\chi_{T}$ is uniformly bounded and equicontinuous. Therefore upon passing to a subsequence
$\chi_{T}\rightarrow\chi$ as $T\rightarrow\infty$, for some continuous function $\chi$. There are two possibilities for each component, either $\chi$ does not vanish identically, or $\chi$ vanishes identically. Boundary components for which $\chi\equiv 0$ will be labeled $i=1,\ldots, i_0$, while those for which $\chi\not\equiv 0$ will be labeled $i=i_0 +1,\ldots,I$

Let $(\bar{\mathbf{M}}\setminus\bar{\mathbf{M}}_{T})_{i}$ be the asymptotically cylindrical end associated with $\Sigma_h^i$, and set $\tilde{\mathbf{M}}_{T}=\bar{\mathbf{M}}
\setminus\cup_{i=1}^{i_0}(\bar{\mathbf{M}}\setminus\bar{\mathbf{M}}_{T})_{i}$. Note that $\tilde{\mathbf{M}}_{T}$ is $\bar{\mathbf{M}}$ minus the asymptotically cylindrical ends corresponding to $\chi\equiv 0$. Consider now the boundary value problem
\begin{align}\label{equation}
\begin{split}
&\Delta_{\bar{\mathbf{g}}_{\delta}}u_{\delta,T}
-\frac{1}{8}\left(\bar{\mathbf{R}}_{\delta}
-2|\bar{\mathbf{E}}_{\delta}|^{2}_{\bar{\mathbf{g}}_{\delta}}
-|\bar{\mathbf{k}}_{\delta}|^{2}_{\bar{\mathbf{g}}_{\delta}}\right)u_{\delta,T}=0
\quad\text{ on }\quad\tilde{\mathbf{M}}_{T},\\
&u_{\delta,T} \rightarrow 1\quad\quad\text{ as }\quad\quad|x|\to\infty,\\
&\partial_{\bar{\nu}}u_{\delta,T}+\frac{1}{4}\bar{H}u_{\delta,T}
=\frac{1}{4}\sqrt{\frac{16\pi}{|\partial_i \bar{\mathbf{M}}_T|_{\tilde{\mathbf{g}}_{\delta,T}}}}u^3_{\delta,T},\quad\quad\text{ on }\quad\quad \partial_i \bar{\mathbf{M}}_T,\quad\quad i=1,\ldots,i_0,\\
&u_{\delta,T}\rightarrow 0\quad\quad\text{ as }\quad\quad x\rightarrow\partial_{i}\bar{\mathbf{M}},\quad\quad i=i_0+1,\ldots,I,
\end{split}
\end{align}
where $\partial_{i}\bar{\mathbf{M}}$ denotes the limiting cross-section within the asymptotic end $(\bar{\mathbf{M}}\setminus\bar{\mathbf{M}}_{T})_{i}$ and $\tilde{\mathbf{g}}_{\delta,T}=u_{\delta,T}^4\bar{\mathbf{g}}_{\delta}$.
The boundary condition above ensures that the mean curvature of the boundary components $\partial_i \bar{\mathbf{M}}_T$, $i=1,\ldots,i_0$ with respect to $\tilde{\mathbf{g}}_{\delta,T}$ is given by $\tilde{H}=\sqrt{\frac{16\pi}{|\partial_i \bar{\mathbf{M}}_T|_{\tilde{\mathbf{g}}_{\delta,T}}}}$, and the equation for $u_{T}$ guarantees that the scalar curvature of the conformal metric is
\begin{equation}\label{conformalscalar}
\tilde{\mathbf{R}}_{\delta,T}
=\left(2|\bar{\mathbf{E}}_{\delta}|^{2}_{\bar{\mathbf{g}}_{\delta}}
+|\bar{\mathbf{k}}_{\delta}|^{2}_{\bar{\mathbf{g}}_{\delta}}\right) u^{-4}_T=2|\tilde{\mathbf{E}}_{\delta,T}|_{\tilde{\mathbf{g}}_{\delta,T}}^2
+|\tilde{\mathbf{k}}_{\delta,T}|_{\tilde{\mathbf{g}}_{\delta,T}}^2,
\end{equation}
where $\tilde{\mathbf{E}}^j_{\delta,T}=u_{\delta,T}^{-4}\bar{\mathbf{E}}_{\delta}^j$ and $\tilde{\mathbf{k}}_{\delta,T}=u_{\delta,T}^{-6}\bar{\mathbf{k}}_{\delta}$. The motivation for imposing different boundary conditions on the two groups of boundary components is to facilitate contributions to the ADM mass from each component. In what follows we will establish the existence of a regular positive solution to \eqref{equation}.

\begin{proposition}\label{solution}
Under the hypotheses of Theorems \ref{thm2.7}-\ref{thm2.12}, if $T$ is sufficiently large and $\delta$ is sufficiently small, there exists a smooth positive solution to the boundary value problem \eqref{equation}.
\end{proposition}

\begin{proof}
The boundary value problem \eqref{equation} corresponds to the Euler-Lagrange equations of the following functional
\begin{align}
\begin{split}
P(v)=&\frac{1}{2}\int_{\tilde{\mathbf{M}}_T}
\left(|\nabla v|^2_{\bar{\mathbf{g}}_{\delta}}
+\frac{1}{8}\left(\bar{\mathbf{R}}_{\delta}
-2|\bar{\mathbf{E}}_{\delta}|^{2}_{\bar{\mathbf{g}}_{\delta}}
-|\bar{\mathbf{k}}_{\delta}|^{2}_{\bar{\mathbf{g}}_{\delta}}\right)(1+v)^2\right)
dV_{\bar{\mathbf{g}}_{\delta}}\\
&-\sum_{i=1}^{i_0}\frac{1}{8}\int_{\partial_i\bar{\mathbf{M}}_T}\bar{H}(1+v)^2
d\bar{A}_{T}
+\sum_{i=1}^{i_0}\frac{\sqrt{\pi}}{2}\left(\int_{\partial_i\bar{\mathbf{M}}_T}(1+v)^4
d\bar{A}_{T}\right)^{1/2},
\end{split}
\end{align}
defined on the space of functions
\begin{equation}
\mathcal{W}=\left\{v\in W^{1,2}_{loc}(\tilde{\mathbf{M}}_T)\text{ }\Big| \text{ } |x|^{j-1}\nabla^j v\in L^2(\bar{\mathbf{M}}_T), \text{ }j=0,1,\text{ } 1+v\in W^{1,2} \left(\tilde{\mathbf{M}}_T \setminus \bar{\mathbf{M}}_{T_0}\right)\right\}.
\end{equation}
Note that by the trace theorem $v\in L^4(\partial \bar{\mathbf{M}}_T)$. Existence of a global minimizer may be achieved through direct methods in the calculus of variations. In particular, since the functional is weakly lower semicontinuous we need only establish that it is coercive. Regularity of the weak solution will follow from standard elliptic theory. Similar problems have been treated in
\cite[Proposition 3.2]{Herzlich} and \cite[Theorem 2.1]{Khuri}.

In order to show coercivity set
\begin{equation}
K_{\delta}=\bar{\mathbf{R}}_{\delta}
-2|\bar{\mathbf{E}}_{\delta}|^{2}_{\bar{\mathbf{g}}_{\delta}}
-|\bar{\mathbf{k}}_{\delta}|^{2}_{\bar{\mathbf{g}}_{\delta}}
-2|\mathbf{X}_{\delta}|^2_{\bar{\mathbf{g}}_{\delta}}
+2\text{div}_{\bar{\mathbf{g}}_{\delta}}\mathbf{X}_{\delta},
\end{equation}
and let $\xi_{T_0}\in C^{\infty}(\bar{\mathbf{M}})$ be a nonnegative cut-off function with its maximum value $\xi_{T_0}\equiv 1$ achieved on $\bar{\mathbf{M}}\setminus\bar{\mathbf{M}}_{T_0}$ and $\xi_{T_0}\equiv 0$ on $\bar{\mathbf{M}}_{T_0 -1}$.
Then according to Lemma \ref{gluing} and its proof
\begin{align}\label{Pin}
\begin{split}	
P(v)=&\frac{1}{2}\int_{\tilde{\mathbf{M}}_T}
\left(|\nabla v|^2_{\bar{\mathbf{g}}_{\delta}}+\frac{1}{8}K_{\delta}(1+v)^2
+\frac{1}{4}\left(|\mathbf{X}_{\delta}|^2_{\bar{\mathbf{g}}_{\delta}}	 -\text{div}_{\bar{\mathbf{g}}_{\delta}}\mathbf{X}_{\delta}\right)(1+v)^2\right)
dV_{\bar{\mathbf{g}}_{\delta}}\\
&-\sum_{i=1}^{i_0}\frac{1}{8}\int_{\partial_i\bar{\mathbf{M}}_T}\bar{H}(1+v)^2
d\bar{A}_{T}
+\sum_{i=1}^{i_0}\frac{\sqrt{\pi}}{2}\left(\int_{\partial_i\bar{\mathbf{M}}_T}(1+v)^4
d\bar{A}_{T}\right)^{1/2}\\
\geq& \int_{\tilde{\mathbf{M}}_T}
\left(\frac{1}{3}|\nabla v|^2_{\bar{\mathbf{g}}_{\delta}}
+\left(\frac{1}{8}(\mu_{EM}-|J|_g)\xi_{T_0}
+\frac{1}{32}|\mathbf{X}_{\delta}|_{\bar{\mathbf{g}}_{\delta}}^2\right)(1+v)^2
-\frac{1}{16}K_{\delta-}(1+v)^2\right)dV_{\bar{\mathbf{g}}_{\delta}}\\ &-\sum_{i=1}^{i_0}\frac{1}{8}\int_{\partial_i\bar{\mathbf{M}}_T}
\left(\bar{H}-X(\bar{\nu})\right)(1+v)^2
d\bar{A}_{T}
+\sum_{i=1}^{i_0}\frac{\sqrt{\pi}}{2}\left(\int_{\partial_i\bar{\mathbf{M}}_T}(1+v)^4
d\bar{A}_{T}\right)^{1/2},
\end{split}
\end{align}
where $K_{\delta}=K_{\delta+}-K_{\delta-}$ with $K_{\delta+}$ and $K_{\delta-}$ representing the nonnegative and nonpositive parts of the function. Note that $K_{\delta-}=0$ except possibly on a set, denoted $\Omega_{\delta}$, of small volume $|\Omega_{\delta}|=O(\delta)$ on which $K_{\delta -}=O(1)$. It then follows from
H\"{o}lder's inequality and the Sobolev inequality \cite[Lemma 3.1]{SchoenYauI} that for large $T$
\begin{align}\label{K-}
\begin{split}
\int_{\tilde{\mathbf{M}}_{T}}K_{\delta-}(1+v)^2dV_{\bar{\mathbf{g}}_{\delta}}
= &\int_{\Omega_{\delta}}K_{\delta-}(1+v)^2dV_{\bar{\mathbf{g}}_{\delta}}\\
\leq &
\left(\int_{\Omega_{\delta}}K_{\delta-}^{3/2}dV_{\bar{\mathbf{g}}_{\delta}}\right)^{2/3}
\left(\int_{\bar{\mathbf{M}}_{T_0}}(1+v)^6 dV_{\bar{\mathbf{g}}_{\delta}}\right)^{1/3}\\
\leq &\delta^{2/3}C_0 \int_{\bar{\mathbf{M}}_{T_0}}|\nabla v|^2_{\bar{\mathbf{g}}_{\delta}}
dV_{\bar{\mathbf{g}}_{\delta}}\\
\leq &\delta^{2/3}C_0 \int_{\tilde{\mathbf{M}}_{T}}|\nabla v|^2_{\bar{\mathbf{g}}_{\delta}}
dV_{\bar{\mathbf{g}}_{\delta}},
\end{split}
\end{align}
where the constant $C_0$ is independent of $\delta$ and $T$.
Next observe that Jensen's inequality gives
\begin{equation}\label{Jensen}
\left(\int_{\partial_i\bar{\mathbf{M}}_T}(1+v)^4
d\bar{A}_{T}\right)^{1/2}\geq
\frac{1}{\sqrt{|\partial_i \bar{\mathbf{M}}_T|}}
\int_{\partial_i\bar{\mathbf{M}}_T}(1+v)^2
d\bar{A}_{T},
\end{equation}
and notice that for each boundary component with labeling $i=1,\ldots,i_0$ the expression $\bar{H}-X(\bar{\nu})$ can be made arbitrarily small by choosing $T$ appropriately large. Therefore for small enough $\delta$ and large enough $T$ we have
\begin{align}\label{Pin1}
\begin{split}	
P(v)\geq& \int_{\tilde{\mathbf{M}}_T}
\left(\frac{1}{4}|\nabla v|^2_{\bar{\mathbf{g}}_{\delta}}
+\left(\frac{1}{8}(\mu_{EM}-|J|_g)\xi_{T_0}
+\frac{1}{32}|\mathbf{X}_{\delta}|_{\bar{\mathbf{g}}_{\delta}}^2\right)
(1+v)^2\right)dV_{\bar{\mathbf{g}}_{\delta}}\\
&
+\left(\frac{1-\vartheta_T}{2}\right)\sum_{i=1}^{i_0}\sqrt{\frac{\pi}{|\partial_i \bar{\mathbf{M}}_T|}}
\int_{\partial_i\bar{\mathbf{M}}_T}(1+v)^2
d\bar{A}_{T},
\end{split}
\end{align}
where $\vartheta_T \rightarrow 0$ as $T\rightarrow\infty$.

Global weighted $L^2$ bounds may be derived from \eqref{Pin1} as follows. On $\tilde{\mathbf{M}}_{T}\setminus\bar{\mathbf{M}}_{T_0}$ the cut-off function $\xi_{T_0}=1$, so that the strict energy condition on the horizon $\mu_{EM}-|J|_g \geq c>0$ together with \eqref{Pin1} gives an $L^2$ estimate for $1+v$.
On $\bar{\mathbf{M}}_{T_0}$, let $r$ be a smooth positive function which coincides with $|x|$ in the asymptotically flat end, then the
weighted Poincar\'{e} inequality \cite[Theorem 1.3]{Bartnik}
\begin{equation}
\int_{\bar{\mathbf{M}}_{T_0}}\frac{v^2}{r^2} dV_{\bar{\mathbf{g}}_{\delta}}
\leq C \int_{\bar{\mathbf{M}}_{T_0}}|\nabla v|_{\bar{\mathbf{g}}_{\delta}}^2 dV_{\bar{\mathbf{g}}_{\delta}},
\end{equation}
together with \eqref{Pin1} yield the desired weighted $L^2$ estimate for $1+v$.
Altogether this produces
\begin{equation}\label{fu}
P(v)\geq C^{-1}\int_{\tilde{\mathbf{M}}_T}\left(|\nabla v|^2_{\bar{\mathbf{g}}_{\delta}}+\xi_{T_0}(1+v)^2
+\frac{(1-\xi_{T_0})}{r^{2}}v^2\right)dV_{\bar{\mathbf{g}}_{\delta}},
\end{equation}
which is the coercivity bound that yields existence.

Let $u_T$ denote the solution of \eqref{equation} produced above, and assume that $u_T$ does not vanish identically on $\partial\tilde{\mathbf{M}}_T$. We will show that the solution is strictly positive. Suppose that $u_T$ is negative somewhere, and let $D_-$ denote the domain on which $u_T<0$. Since $u_T\rightarrow 1$ as $|x|\rightarrow \infty$, the closure of $D_- \cap \bar{\mathbf{M}}_{T}$ must be compact. Now multiply \eqref{equation} by $u_T$, integrate by parts, and apply the same techniques used to derive \eqref{Pin1} to find
\begin{align}\label{7890}
\begin{split}
&\int_{D_-}\frac{1}{4}|\nabla u_{\delta,T}|^2_{\bar{\mathbf{g}}_{\delta}}
dV_{\bar{\mathbf{g}}_{\delta}}\\
\leq &
-\int_{D_-}\frac{1}{8}K_{\delta}u_{\delta,T}^2 dV_{\bar{\mathbf{g}}_{\delta}}
-\sum_{i=1}^{i_0}\int_{D_-\cap\partial_{i}\bar{\mathbf{M}}_T}
\left(u_{\delta,T}\partial_{\bar{\nu}}u_{\delta,T}
+\frac{1}{4}X(\bar{\nu})u_{\delta,T}^2\right)d\bar{A}_T\\
\leq & \int_{D_-}\frac{1}{8}K_{\delta -}u_{\delta,T}^2 dV_{\bar{\mathbf{g}}_{\delta}}
-\sum_{i=1}^{i_0}\frac{1}{4}\int_{D_-\cap\partial_{i}\bar{\mathbf{M}}_T}
\left(\sqrt{\frac{16\pi}{|\partial_i \bar{\mathbf{M}}_T|_{\tilde{\mathbf{g}}_{\delta}}}}u^4_{\delta,T}
+\left(X(\bar{\nu})-\bar{H}\right)u_{\delta,T}^2\right)d\bar{A}_T\\
\leq &\delta^{2/3}C_0 \int_{D_-}|\nabla u_{\delta,T}|^2_{\bar{\mathbf{g}}_{\delta}}
dV_{\bar{\mathbf{g}}_{\delta}}
\end{split}
\end{align}
Selecting $\delta$ sufficiently small then produces the contradiction
\begin{equation}\label{kjh}
\int_{D_-}|\nabla u_{\delta,T}|^2_{\bar{\mathbf{g}}_{\delta}}
dV_{\bar{\mathbf{g}}_{\delta}}\leq 0.
\end{equation}
Therefore $u_{\delta,T}\geq 0$, and the strict positivity follows from Hopf's maximum principle.

It remains to show that $u_{\delta,T}$ does not identically vanish on $\partial\tilde{\mathbf{M}}_{T}$.
Following \cite[Proposition 3.2]{Herzlich} consider the linear boundary value problem
\begin{align}
\begin{split}
&\Delta_{\bar{\mathbf{g}}_{\delta}}z
-\frac{1}{8}\left(\bar{\mathbf{R}}_{\delta}
-2|\bar{\mathbf{E}}_{\delta}|^{2}_{\bar{\mathbf{g}}_{\delta}}
-|\bar{\mathbf{k}}_{\delta}|^{2}_{\bar{\mathbf{g}}_{\delta}}\right)z=0
\quad\text{ on }\quad\tilde{\mathbf{M}}_{T},\\
&z \rightarrow 0\quad\quad\text{ as }\quad\quad|x|\to\infty,\\
&\partial_{\bar{\nu}}z
=-1,\quad\quad\text{ on }\quad\quad \partial_i \bar{\mathbf{M}}_T,\quad\quad i=1,\ldots,i_0,\\
&z\rightarrow 0\quad\quad\text{ as }\quad\quad x\rightarrow\partial_{i}\bar{\mathbf{M}},\quad\quad i=i_0+1,\ldots,I.
\end{split}
\end{align}
The methods used to establish \eqref{fu} can be applied here to show that this boundary value problem has no kernel, and hence a unique smooth solution exists. Furthermore, in analogy with \eqref{kjh} it can be shown that $z$ is strictly positive. In particular, expanding in spherical harmonics in the asymptotically flat end yields $z = a/r +O_1(r^{-2})$ for some constant $a>0$. Suppose that $u_{\delta,T}=1+v_{\delta,T}$ vanishes identically on $\partial\tilde{\mathbf{M}}_{T}$, then
\begin{align}
\begin{split}
&P(v_{\delta,T} +\varepsilon z)-P(v_{\delta,T})\\
=& \frac{1}{2}\int_{\tilde{\mathbf{M}}_T}\left(|\nabla(v_{\delta,T}+\varepsilon z)|^2_{\bar{\mathbf{g}}_{\delta}}-|\nabla v_{\delta,T}|^2_{\bar{\mathbf{g}}_{\delta}}\right) dV_{\bar{\mathbf{g}}_{\delta}}+O(\varepsilon^2)\\
&+\frac{1}{16}\int_{\tilde{\mathbf{M}}_T}\left(\bar{\mathbf{R}}_{\delta}
-2|\bar{\mathbf{E}}_{\delta}|^{2}_{\bar{\mathbf{g}}_{\delta}}
-|\bar{\mathbf{k}}_{\delta}|^{2}_{\bar{\mathbf{g}}_{\delta}}\right)
\left((1+v_{\delta,T}+\varepsilon z)^2-(1+v_{\delta,T})^2\right)dV_{\bar{\mathbf{g}}_{\delta}}\\
=&\int_{\tilde{\mathbf{M}}_T}\varepsilon(1+v_{\delta,T})
\left(-\Delta_{\bar{\mathbf{g}}_{\delta}}z
+\frac{1}{8}\left(\bar{\mathbf{R}}_{\delta}
-2|\bar{\mathbf{E}}_{\delta}|^{2}_{\bar{\mathbf{g}}_{\delta}}
-|\bar{\mathbf{k}}_{\delta}|^{2}_{\bar{\mathbf{g}}_{\delta}}\right)z\right)
dV_{\bar{\mathbf{g}}_{\delta}}\\
&+\int_{S_{\infty}}\varepsilon(1+v_{\delta,T})\partial_r z +O(\varepsilon^2)\\
=&-4\pi\varepsilon a +O(\varepsilon^2),
\end{split}
\end{align}
where $S_{\infty}$ represents the limit of coordinate spheres in the asymptotic end as $r\rightarrow\infty$. For $\varepsilon$ small enough this yields $P(v_{\delta,T}+\varepsilon z)<P(v_{\delta,T})$, which contradicts the minimizing property of $v_{\delta,T}$.
\end{proof}

The solution $u_{\delta,T}$ of \eqref{equation} will be used to obtain lower bounds for the mass $\bar{\mathbf{m}}$ of the glued manifold $(\bar{\mathbf{M}},\bar{\mathbf{g}})$. Expanding in spherical harmonics gives the expression
\begin{equation}
u_{\delta,T}=1+\frac{\mathcal{A}_{\delta,T}}{|x|}+O(|x|^{-2})\quad\quad\text{ as }\quad|x|\rightarrow\infty,
\end{equation}
for some constant $\mathcal{A}_{\delta,T}$. The mass of $\tilde{\mathbf{g}}_{\delta,T}=u_{\delta,T}^4 \bar{\mathbf{g}}_{\delta}$ is then given by
\begin{equation}\label{PAjj}
\tilde{\mathbf{m}}_{\delta,T}=\bar{\mathbf{m}}+2\mathcal{A}_{\delta,T}.
\end{equation}
Moreover
\begin{equation}\label{PA} P(v_{\delta,T})=\frac{1}{2}\int_{S_{\infty}}u_{\delta,T}\partial_{r}u_{\delta,T}
=-2\pi\mathcal{A}_{\delta,T},
\end{equation}
so that lower bounds for $\bar{\mathbf{m}}$ may be obtained from lower bounds for $P(v_{\delta,T})$ and $\tilde{\mathbf{m}}_{\delta,T}$. In order to facilitate this we note that the sequence of solutions to \eqref{equation} subconverges to a solution of
\begin{align}\label{equation1}
\begin{split}
&\Delta_{\bar{\mathbf{g}}}u
-\frac{1}{8}\left(\bar{\mathbf{R}}
-2|\bar{\mathbf{E}}|^{2}_{\bar{\mathbf{g}}}
-|\bar{\mathbf{k}}|^{2}_{\bar{\mathbf{g}}}\right)u=0
\quad\text{ on }\quad\bar{\mathbf{M}},\\
&u \rightarrow 1\quad\text{ as }\quad|x|\to\infty,\quad\quad
u\rightarrow 0\quad\text{ as }\quad x\rightarrow\partial_{i}\bar{\mathbf{M}},\quad\quad i=1,\ldots,I.
\end{split}
\end{align}
Although the coefficients of this equation are not continuous, the divergence structure present in the scalar curvature will help to mitigate the singular behavior of the solution.

\begin{lemma}\label{limitsolution}
There exists a subsequence $\{u_{\delta_l,T_l}\}_{l=0}^{\infty}$ of the solutions produced in Proposition \ref{solution} which converges in $C^0(\bar{\mathbf{M}})\cap C^{2}(\bar{\mathbf{M}}\setminus\Sigma)$, and weakly in $W^{1,2}_{loc}(\bar{\mathbf{M}})$, to a positive solution $u\in C^{0,\alpha}(\bar{\mathbf{M}})\cap C^{\infty}(\bar{\mathbf{M}}\setminus\Sigma)$ of \eqref{equation1} for some $\alpha\in(0,1)$. Furthermore $u$ is not constant on $\Omega$.
\end{lemma}

\begin{proof}
First observe that $\tilde{\mathbf{m}}_{\delta,T}\geq 0$. To see this note that $\tilde{\mathbf{g}}_{\delta,T}$ is of nonnegative scalar curvature from \eqref{conformalscalar}, and according to the boundary condition of \eqref{equation} each component of $\partial\tilde{\mathbf{M}}_{T}$ has zero Hawking mass. Furthermore, a strict energy condition at the apparent horizon implies that $u_{\delta,T}$ decays exponentially along the asymptotically cylindrical ends of $\tilde{\mathbf{M}}_{T}$ (see \cite{SchoenYauII}). Therefore the asymptotically cylindrical ends are conformally closed, and one could either start the inverse mean curvature flow at such a closed end or at a boundary component of zero Hawking mass in order to find nonnegativity of the mass.

Equation \eqref{PA} now yields the upper bound $P(v_{\delta,T})\leq \pi\bar{\mathbf{m}}$, which is independent of $\delta$ and $T$. We may then use
\eqref{fu} to find
\begin{equation}
\int_{\tilde{\mathbf{M}}_T}\left(|\nabla u_{\delta,T}|^2_{\bar{\mathbf{g}}}+\xi_{T_0}u_{\delta,T}^2
+\frac{(1-\xi_{T_0})}{r^{2}}(u_{\delta,T}-1)^2\right)dV_{\bar{\mathbf{g}}}
\leq C\bar{\mathbf{m}},
\end{equation}
which yields a uniform $W^{1,2}$ estimate on compact subsets. By taking an exhausting sequence of domains and using a diagonal argument, a subsequence $u_{\delta_l,T_l}$ is obtained which converges weakly in $W^{1,2}_{loc}(\bar{\mathbf{M}})$ to $u$. Next, consider the divergence structure present in \eqref{equation}, namely this equation may be rewritten as
\begin{equation}\label{divergencestructure}
\text{div}_{\bar{\mathbf{g}}_{\delta}}\left(
\nabla u_{\delta,T}
+\frac{1}{4}\mathbf{X}_{\delta} u_{\delta,T}\right)
-\frac{1}{4}\mathbf{X}_{\delta}\cdot\nabla u_{\delta,T}
-\frac{1}{8}\left(K_{\delta}-2|\mathbf{X}_{\delta}|^2_{\bar{\mathbf{g}}_{\delta}}\right)
u_{\delta,T}=0.
\end{equation}
Since the coefficients $\mathbf{X}_{\delta}$ and $K_{\delta}$ are uniformly bounded, this allows for a notion of weak solution in which the coefficients converge in $L^{2}_{loc}$. The weak convergence $u_{\delta_l,T_l}\rightharpoonup u$ then implies that $u\in W^{1,2}_{loc}(\bar{\mathbf{M}})$ is a weak solution of \eqref{equation1}. Standard elliptic regularity guarantees that $u\in C^{\infty}(\bar{\mathbf{M}}\setminus\Sigma)$, and also that the convergence is in the $C^2$-topology on $\bar{\mathbf{M}}\setminus\Sigma$.

Refined regularity for the limit can be obtained as follows. Let $D\subset D'$ be compact in $\bar{\mathbf{M}}$. The divergence structure of \eqref{divergencestructure} allows for an application of the De Giorgi-Nash-Moser estimate \cite[Theorem 8.24]{GilbargTrudinger}
\begin{equation}\label{degiorgi}
\parallel u_{\delta,T}\parallel_{C^{0,\alpha}(D)}\leq C\parallel u_{\delta,T}\parallel_{L^{2}(D')},
\end{equation}
where $\alpha\in(0,1)$ and $C$ are independent of $\delta$ and $T$. Since the relevant subsequence of $u_{\delta,T}$ converges weakly in $W^{1,2}$, it converges strongly in $L^2$, and hence \eqref{degiorgi} implies that the limit $u\in C^{0,\alpha}$. In particular $u$ is uniformly bounded along the asymptotically cylindrical ends of $\bar{\mathbf{M}}$, which implies exponential decay to zero and shows that it satisfies the asymptotic boundary condition of \eqref{equation1}.

Lastly it will be shown that $u$ is strictly positive and nonconstant on $\Omega$. As the coefficients in the divergence structure \eqref{divergencestructure} are uniformly bounded, the Harnack inequality \cite[Corollary 8.21]{GilbargTrudinger} holds with a constant $C$ independent of $\delta$ and $T$, that is
\begin{equation}
\sup_{D} u_{\delta,T}\leq C\inf_{D} u_{\delta,T}.
\end{equation}
Since $u_{\delta,T}$ subconverges to $u$ in $C^0$, this same inequality holds for $u$. Thus if $u$ vanishes somewhere it must vanish everywhere, which contradicts the fact that $u\rightarrow 1$ as $|x|\rightarrow\infty$, so that $u$ is positive. Lastly if $u$ is constant in $\Omega$ then it must be zero there in light of the exponential decay along the asymptotically cylindrical ends, but this again contradicts the behavior in the asymptotically flat end.
\end{proof}

Recall that in the setting of Theorems \ref{thm2.7}-\ref{thm2.12}, the boundary of $\Omega$ consists of two parts $\partial\Omega=\Sigma_h \cup\Sigma$, an inner horizon piece and an outer untrapped piece. We are now in a position to obtain a lower bound for the mass $\bar{\mathbf{m}}$ in terms of the horizon area. To do this we define the following constant that appears in the above mentioned theorems
\begin{equation}\label{definitiongamma}
\gamma=\left(\sum_{i=1}^I\sqrt{4\pi|\Sigma_h^i|}\right)^{-1}
\parallel \nabla u\parallel_{L^2(\bar{\Omega},\bar{g})}^2.
\end{equation}
According to Lemma \ref{limitsolution} this constant is nonzero and finite, and is invariant under rescalings of the metric.

\begin{proposition}\label{lemma5.5}
Under the hypotheses of Theorems \ref{thm2.7}-\ref{thm2.12}
\begin{equation}
\bar{\mathbf{m}}\geq \lim_{l\to\infty}\tilde{\mathbf{m}}_{\delta_l,T_l}+\frac{\gamma}{1+\gamma}
\sum_{i=1}^{I}\sqrt{\frac{|\Sigma_{h}^i|}{4\pi}},
\end{equation}
\end{proposition}

\begin{proof}
From \eqref{PAjj} and \eqref{PA} we have
\begin{equation}
\bar{\mathbf{m}}=\tilde{\mathbf{m}}_{\delta,T}+\frac{1}{\pi}P(v_{\delta,T}).
\end{equation}
Moreover \eqref{Pin1} and \cite[Lemma 2.2]{Khuri} imply that
\begin{equation}
P(v_{\delta,T})\geq \frac{1}{4}\int_{\tilde{\mathbf{M}}_T}
|\nabla u_{\delta,T}|^2_{\bar{\mathbf{g}}_{\delta}}
dV_{\bar{\mathbf{g}}_{\delta}}
+\left(\frac{1-\vartheta_T}{2}\right)\sum_{i=1}^{I}\sqrt{\frac{\pi}{|\partial_i \bar{\mathbf{M}}_T|}}
\int_{\partial_i\bar{\mathbf{M}}_T}u_{\delta,T}^2
d\bar{A}_{T}.
\end{equation}
Following \cite[Section 3]{Khuri} then yields
\begin{equation}\label{Pin3}
P(v_{\delta,T})\geq\frac{\gamma_{\delta,T}(1-\vartheta_{T})}
{2\left(1+\gamma_{\delta,T}\right)}\sum_{i=1}^I
\sqrt{\pi|\partial_i\bar{\mathbf{M}}_T|_{\bar{\mathbf{g}}_{\delta}}},
\end{equation}
where
\begin{equation}
\gamma_{\delta,T}=\frac{\int_{\bar{\mathbf{M}}_T\cap \Omega}|\nabla u_{\delta,T}|^2_{\bar{\mathbf{g}}_{\delta}}dV_{\bar{\mathbf{g}}_{\delta}}}
{2\left(1-\vartheta_{T}\right)\sum_{i=1}^I
\sqrt{\tfrac{\pi}{|\partial_i\bar{\mathbf{M}}_T|_{\bar{\mathbf{g}}_{\delta}}}}
\int_{\partial_i\bar{\mathbf{M}}_T}v_{\delta,T}^2 d\bar{A}_{T}}.
\end{equation}
Suppose that
\begin{equation}\label{lop}
\int_{\partial_i\bar{\mathbf{M}}_{T_l}}v_{\delta_l,T_l}^2 d\bar{A}_{T_l}\rightarrow
|\Sigma_h^i|,
\end{equation}
where $u_{\delta_l,T_l}$ is the subsequence that converges to $u$ weakly in $W^{1,2}$. Since the Hilbert space norm is weakly lower semicontinuous, we then find that
\begin{equation}
\lim_{l\rightarrow\infty}\gamma_{\delta_l,T_l}\geq\gamma,
\end{equation}
where we have also used that $|\partial_i\bar{\mathbf{M}}_{T_l}|_{\bar{\mathbf{g}}_{\delta_l}}\rightarrow|\Sigma_h^i|$.
Furthermore, since the function $\gamma\rightarrow\tfrac{\gamma}{1+\gamma}$ is monotonically nondecreasing the desired result follows.

It remains to show that \eqref{lop} is valid. Along the asymptotic ends $(\bar{\mathbf{M}}\setminus\bar{\mathbf{M}}_{T})_i$, $i=i_0 +1,\ldots,I$ this is clear as $u_{\delta,T}\rightarrow 0$ along these ends. For the remaining ends, choose $T_0<T$ and perform the integration by parts, on $\bar{\mathbf{M}}_{T}\setminus\bar{\mathbf{M}}_{T_0}$, that led to \eqref{7890} to obtain
\begin{align}\label{11112}
\begin{split}
&\int_{\left(\bar{\mathbf{M}}_{T}\setminus\bar{\mathbf{M}}_{T_0}\right)_i}
\frac{1}{4}|\nabla u_{\delta,T}|^2_{\bar{\mathbf{g}}_{\delta}}
dV_{\bar{\mathbf{g}}_{\delta}}
+\left(\int_{\partial_{i}\bar{\mathbf{M}}_T}\pi u_{\delta,T}^4 d\bar{A}_{T}\right)^{1/2}\\
\leq &
\int_{\partial_{i}\bar{\mathbf{M}}_{T_0}}
\left(u_{\delta,T}\partial_{\bar{\nu}}u_{\delta,T}
+\frac{1}{4}X(\bar{\nu})u_{\delta,T}^2\right)d\bar{A}_{T_0}
+\frac{1}{4}\int_{\partial_{i}\bar{\mathbf{M}}_T}
\left(\bar{H}-X(\bar{\nu})\right)u_{\delta,T}^2 d\bar{A}_T.
\end{split}
\end{align}
Since $|\bar{H}|+|X(\bar{\nu})|\rightarrow 0$ as $T\rightarrow \infty$ for $i=1,\ldots,i_0$, Jensen's inequality implies that
\begin{equation}\label{gthy}
\int_{\partial_{i}\bar{\mathbf{M}}_T} u_{\delta,T}^2 d\bar{A}_{T}
\leq C\int_{\partial_{i}\bar{\mathbf{M}}_{T_0}}
\left(u_{\delta,T}|\partial_{\bar{\nu}}u_{\delta,T}|
+|X(\bar{\nu})|u_{\delta,T}^2\right)d\bar{A}_{T_0}
\end{equation}
where $C$ is independent of $T$. Since $u_{\delta_l,T_l}\rightarrow u$ smoothly on $\partial_{i}\bar{\mathbf{M}}_{T_0}$, and $u$ along with its derivatives have exponential decay along the asymptotically cylindrical ends, the right-hand side of \eqref{gthy} can be made arbitrarily small for the sequence $u_{\delta_l,T_l}$ by letting $l\rightarrow\infty$ and choosing $T_0$ sufficiently large. The desired conclusion \eqref{lop} now follows.
\end{proof}

\section{Penrose-Like Inequalities for the Liu-Yau and Wang-Yau Masses}
\label{sec7}\setcounter{equation}{0}
\setcounter{section}{7}

\begin{proof}[Proof of Theorem \ref{thm2.7}]
As explained in Section \ref{sec3}, the Liu-Yau mass is bounded below by the mass of the glued manifold. Therefore Proposition \ref{lemma5.5} implies that
\begin{equation}\label{tytyty}
m_{LY}(\Sigma)\geq \bar{\mathbf{m}}
\geq\lim_{l\to\infty}\tilde{\mathbf{m}}_{\delta_l,T_l}+\frac{\gamma}{1+\gamma}
\sum_{i=1}^{I}\sqrt{\frac{|\Sigma_{h}^i|}{4\pi}}
\geq \frac{\gamma}{1+\gamma}
\sqrt{\frac{|\Sigma_{h}|}{4\pi}},
\end{equation}
where we have used the fact that $\tilde{\mathbf{m}}_{\delta,T}\geq 0$ as explained in the proof of Lemma \ref{limitsolution}; this establishes \eqref{LY1in}. In order to show \eqref{LY2in}, use the above sequence of
inequalities combined with the area-angular momentum-charge inequality
\cite{BrydenKhuri,ClementJaramilloReiris} which holds under the state hypotheses for stable apparent horizons:
\begin{equation}
|\Sigma_h^i|\geq4\pi\sqrt{(Q_h^i)^4+4(\mathcal{J}_h^i)^2}.
\end{equation}
\end{proof}

\begin{proof}[Proof of Theorem \ref{thm2.9}]
We will estimate $\tilde{\mathbf{m}}_{\delta_l,T_l}$ from below and apply
\eqref{tytyty} to achieve \eqref{LY3in}. The estimate of the conformal mass will be via inverse mean curvature flow. First observe that $\Omega$ is simply connected. To see this, note that the assumption of no interior apparent horizons (including no immersed MOTS) together with an untrapped outer boundary $\Sigma$ (of spherical topology) allows for an application of Theorem 5.1 of \cite[Theorem 5.1]{EichmairGallowayPollack} which yields the desired conclusion. Consider a weak IMCF $\{\tilde{S}_{t}^{\delta,T}\}_{t=0_*}^{\infty}$ within $(\tilde{\mathbf{M}}_{T},\tilde{\mathbf{g}}_{\delta,T})$. This manifold has either a single component inner boundary of zero Hawking mass, or the `point at infinity' in the conformally closed remnant of an asymptotically cylindrical end, from which the IMCF will emanate. In the former case $0_*=0$ while in the latter case $0_* =-\infty$. By simple connectivity each leaf of the flow is connected \cite[Lemma 4.2]{HuiskenIlmanen}. Furthermore the leaf of largest area contained within $\Omega$, and outside of the $\delta$-tubular neighborhood of $\Sigma$ where $\bar{\mathbf{g}}_{\delta}\neq\bar{\mathbf{g}}$, will be denoted by $\tilde{S}_{\tilde{t}_{0}}^{\delta,T}$. Observe that the flow is smooth for a nonzero amount of time within $\Omega$, in other words it does not instantaneously jump from the initial time to a surface intersecting $\Sigma$ or lying outside of $\Omega$. This is due to the fact that the flow is by outermost minimal area enclosures, and the initial surface has area $|\tilde{S}_{0_*}^{\delta,T}|_{\tilde{\mathbf{g}}_{\delta,T}}$ which can be made arbitrarily small compared to $|\Sigma|_{\tilde{\mathbf{g}}_{\delta,T}}$ when $0_*=0$ as a result of \eqref{11112}, \eqref{gthy},
while $|\tilde{S}_{0_*}^{\delta,T}|_{\tilde{\mathbf{g}}_{\delta,T}}=0$ when $0_*=-\infty$. According to Geroch monotonicity \cite{HuiskenIlmanen} and the scalar curvature formula \eqref{conformalscalar} we have
\begin{align}\label{102938}
\begin{split}
\tilde{\mathbf{m}}_{\delta,T}\geq&\frac{1}{(16\pi)^{3/2}}
\int_{0_*}^{\infty}\int_{\tilde{S}_t^{\delta,T}}
\sqrt{|\tilde{S}_t^{\delta,T}|_{\tilde{\mathbf{g}}_{\delta,T}}} \tilde{\mathbf{R}}_{\delta,T}
d\tilde{A}_{t}dt\\
\geq&\frac{1}{(16\pi)^{3/2}}
\int_{0_*}^{\tilde{t}_0}\int_{\tilde{S}_t^{\delta,T}}
\sqrt{|\tilde{S}_t^{\delta,T}|_{\tilde{\mathbf{g}}_{\delta,T}}} \left(2|\tilde{\mathbf{E}}_{\delta,T}|_{\tilde{\mathbf{g}}_{\delta,T}}^2
+|\tilde{\mathbf{k}}_{\delta,T}|_{\tilde{\mathbf{g}}_{\delta,T}}^2\right)
d\tilde{A}_{t}dt\\
=&\frac{1}{(16\pi)^{3/2}}
\int_{0_*}^{\tilde{t}_0}\int_{\tilde{S}_t^{\delta,T}}
\sqrt{|\tilde{S}_t^{\delta,T}|_{\tilde{\mathbf{g}}_{\delta,T}}} \left(2|\bar{\mathbf{E}}_{\delta}|_{\bar{\mathbf{g}}_{\delta}}^2
+|\bar{\mathbf{k}}_{\delta}|_{\bar{\mathbf{g}}_{\delta}}^2\right)
d\bar{A}_{t}dt\\
:=&\mathcal{I}_{Q}(\delta,T)+\mathcal{I}_{\mathcal{J}}(\delta,T).
\end{split}
\end{align}

In what follows we will estimate each of the integrals $\mathcal{I}_{Q}(\delta,T)$ and $\mathcal{I}_{\mathcal{J}}(\delta,T)$ separately. Note that in the domain associated with these two integrals $\bar{\mathbf{g}}_{\delta}=\bar{g}$, $\bar{\mathbf{E}}_{\delta}=\bar{E}$, $\bar{\mathbf{k}}_{\delta}=\bar{k}$, and $\tilde{\mathbf{g}}_{\delta,T}=\tilde{g}_{\delta,T}:=u_{\delta,T}^4 \bar{g}$.
Consider first the electric field term. Using the divergence free property \eqref{barEproperties} and equality of charges \eqref{kjkl} produces
\begin{align}
\begin{split}
\mathcal{I}_{Q}(\delta,T)=&\frac{2}{(16\pi)^{3/2}}
\int_{0_*}^{\tilde{t}_0}
\sqrt{|\tilde{S}_t^{\delta,T}|_{\tilde{g}_{\delta,T}}}
\left(\int_{\tilde{S}_t^{\delta,T}}
|\bar{E}|_{\bar{g}}^2 d\bar{A}_t\right) dt\\
\geq& \frac{2}{(16\pi)^{3/2}}
\int_{0_*}^{\tilde{t}_0}
\frac{\sqrt{|\tilde{S}_t^{\delta,T}|_{\tilde{g}_{\delta,T}}}}
{|\tilde{S}_t^{\delta,T}|_{\bar{g}}}\left(\int_{\tilde{S}_t^{\delta,T}}
\bar{E}(\bar{\nu}) d\bar{A}_{t}\right)^2 dt\\
\geq&\frac{2(4\pi Q)^2\sqrt{|\tilde{S}_{\tilde{t}_0}^{\delta,T}|_{\tilde{g}_{\delta,T}}}}
{(16\pi)^{3/2}
\sup_{0_*< t\leq \tilde{t}_0}|\tilde{S}_t^{\delta,T}|_{\bar{g}}}
\int_{0_*}^{\tilde{t}_0}e^{\frac{(t-\tilde{t}_0)}{2}}dt\\
=&\frac{Q^2
\sqrt{\pi|\tilde{S}_{\tilde{t}_0}^{\delta,T}|_{\tilde{g}_{\delta,T}}}}
{\sup_{0_*< t\leq \tilde{t}_0}|\tilde{S}_t^{\delta,T}|_{\bar{g}}}
\left(1-e^{\frac{(0_* -\tilde{t}_0)}{2}}\right).
\end{split}
\end{align}
Due to the convergence of Lemma \ref{limitsolution}, in analogy with the proof of Theorem \ref{thmBYchargein} a subsequence of $\{\tilde{S}_{t}^{\delta_l,T_l}\}_{t=0_*}^{\tilde{t}_0}$ (denoted with the same notation) converges to a weak IMCF $\{\tilde{S}_{t}\}_{t=-\infty}^{t_0}$ within $(\Omega,\tilde{g}=u^4 \bar{g})$. Note that as in \cite[Lemma 8.1]{HuiskenIlmanen}, an appropriate translation downwards of the sequence of level set functions defining the weak flows is needed to obtain convergence.
Moreover since
\begin{equation}
|\tilde{S}^{\delta_l ,T_l}_{0}|_{\tilde{g}_{\delta_l,T_l}}=e^{-\tilde{t}_0}|\tilde{S}^{\delta_l ,T_l}_{\tilde{t}_0}|_{\tilde{g}_{\delta_l , T_l}},
\end{equation}
and
\begin{equation}
|\tilde{S}^{\delta_l ,T_l}_{0}|_{\tilde{g}_{\delta_l,T_l}}\to 0, \quad\quad\quad |\tilde{S}^{\delta_l ,T_l}_{\tilde{t}_0}|_{\tilde{g}_{\delta_l , T_l}}\rightarrow
|\tilde{S}_{t_0}|_{\tilde{g}}\neq 0 \quad\text{ as }\quad l\to\infty,
\end{equation}
it follows that $\tilde{t}_0=\tilde{t}_0 (l)\rightarrow \infty$. Therefore
\begin{equation}\label{qi}
\lim_{l\rightarrow\infty}\mathcal{I}_{Q}(\delta_l,T_l)\geq
\lambda\sqrt{\frac{\pi}{|\Sigma_h|}}Q^2.
\end{equation}
where
\begin{equation}\label{lambdadefinition}
\lambda=\frac{\sqrt{|\tilde{S}_{t_0}|_{\tilde{g}}
|\Sigma_h|}}{\sup_{-\infty< t\leq t_0}|\tilde{S}_t|_{\bar{g}}}.
\end{equation}

Consider now the angular momentum contribution. In this setting axisymmetry is assumed with $\eta$ denoting the rotational Killing field. The hypothesis \eqref{frobenius} combined with \cite[Appendix]{JaraczKhuri} shows that $h(\eta,\bar{\nu})=0$ on any axisymmetric surface $S\in(\Omega,\bar{g})$ with unit normal $\bar{\nu}$, in particular for the leaves of an IMCF emanating from an axisymmetric surface; this implies that $\bar{k}(\eta,\bar{\nu})=k(\eta,\bar{\nu})$. In addition, observe that the relation between area elements and surface normals with respect to the two metrics $g$ and $\bar{g}$ are given by
\begin{equation}
d\bar{A}=\sqrt{1+|\nabla f|_{S}|_g^2}dA, \quad\quad\quad
\overline{\nu}_i=\frac{\nu_i}{\sqrt{\overline{g}^{ij}\nu_i\nu_j}}
=\sqrt{\frac{1+|\nabla f|^2}{1+|\nabla f|_{S}|^2}}\nu_i.
\end{equation}
It follows that
\begin{align}
\begin{split}
(8\pi \mathcal{J})^2=&\left(\int_{S}k^{ij}\nu_i\eta_j dA\right)^2\\
=&\left(\int_{S}k^{ij}\overline{\nu}_i \eta_j
\sqrt{\frac{1+|\nabla f|_{S}|^2}{1+|\nabla f|^2}} dA\right)^2\\
\leq&\left(\int_{S}|k(\overline{\nu},\eta)|
d\bar{A}\right)^2\\
=&\left(\int_{S}|k(\overline{\nu},\eta)-h(\overline{\nu},\eta)|
d\bar{A}\right)^2\\
\leq&\left(\int_{S}|k-h|_{\bar{g}}
|\eta|d\bar{A}\right)^2\\
\leq&\int_{S}|k-h|_{\bar{g}}^2 d\bar{A}
\int_{S}|\eta|^2 d\bar{A}.
\end{split}
\end{align}
We then have
\begin{align}\label{a1'}
\begin{split}
\mathcal{I}_{\mathcal{J}}(\delta,T)=&\frac{1}{(16\pi)^{3/2}}
\int_{0_*}^{\tilde{t}_0}\sqrt{|\tilde{S}^{\delta,T}_t|_{\tilde{g}_{\delta,T}}}
\left(\int_{\tilde{S}^{\delta,T}_t}
|k-h|^2_{\bar{g}} d\bar{A}_{t}\right)dt\\
\geq&\frac{(8\pi \mathcal{J})^2}{(16\pi)^{3/2}}
\int_{0_*}^{\tilde{t}_0}\frac{\sqrt{|\tilde{S}^{\delta,T}_t|_{\tilde{g}_{\delta,T}}}}
{\int_{\tilde{S}^{\delta,T}_t}|\eta|^2 d\bar{A}_{t}} dt\\
\geq&
\frac{(8\pi \mathcal{J})^2}{(16\pi)^{3/2}\max_{\Omega}|\eta|^2}
\int_{0_*}^{\tilde{t}_0}
\frac{\sqrt{|\tilde{S}^{\delta,T}_{t}|_{\tilde{g}_{\delta,T}}}}
{|\tilde{S}^{\delta,T}_t|_{\bar{g}}} dt
\\
=&\frac{2(8\pi \mathcal{J})^2}{(16\pi)^{3/2}\max_{\Omega}|\eta|^2}
\frac{\sqrt{|\tilde{S}^{\delta,T}_{\tilde{t}_0}|_{\tilde{g}_{\delta,T}}}}{\sup_{0_*< t\leq \tilde{t}_0}|\tilde{S}^{\delta,T}_t|_{\bar{g}}}\left(1-e^{\frac{(0_* -\tilde{t}_0)}{2}}\right),
\end{split}
\end{align}
and hence
\begin{equation}\label{ji}
\lim_{l\rightarrow\infty}\mathcal{I}_{\mathcal{J}}(\delta_l,T_l)\geq
\frac{4\pi^2\lambda}{\mathcal{C}^2}\sqrt{\frac{4\pi}{|\Sigma_h|}}\mathcal{J}^2.
\end{equation}

Lastly, combining \eqref{tytyty}, \eqref{102938}, \eqref{qi}, and \eqref{ji} produces
\begin{equation}
m_{LY}(\Omega)\geq \frac{\gamma}{1+\gamma}\sqrt{\frac{|\Sigma_h|}{4\pi}}
+\lambda\sqrt{\frac{\pi}{|\Sigma_h|}}Q^2
+\frac{4\pi^2 \lambda}{\mathcal{C}^2}\sqrt{\frac{4\pi}{|\Sigma_h|}}\mathcal{J}^2.
\end{equation}
The desired inequality \eqref{LY3in} is then achieved by squaring both sides and rearranging terms.
\end{proof}

\begin{proof}[Proof of Theorems \ref{thm2.11} and \ref{thm2.12}]
This is a straightforward combination of the Wang-Yau proof \cite{WangYauPMT1}, and the proof of Theorems \ref{thm2.7} and \ref{thm2.9}.
\end{proof}

\section{Penrose-Like Inequality for Quasi-Local Mass With a Static Reference}\label{sec8}\setcounter{equation}{0}
\setcounter{section}{8}

Here we show that the inequalities established in previous sections can be extended to quasi-local masses having a static reference spacetime other than Minkowski space. Recall that a 4-dimensional static spacetime is a warped product $(\mathbb{R}\times \mathbb{M}^3,-V^2 dt^2 +\mathrm{g})$, in which $V\in C^{\infty}(\mathbb{M}^3)$
is positive except on the (possibly empty) boundary $\partial\mathbb{M}^3$ where it vanishes, and $\mathrm{g}$ is a Riemannian metric on $\mathbb{M}^3$. The triple $(\mathbb{M}^3,\mathrm{g},V)$ is referred to as a static manifold.
If the static spacetime is vacuum then the potential $V$ satisfies the static equations
\begin{equation}
\left(\Delta_{\mathrm{g}}V\right) \mathrm{g}-\text{Hess}_{\mathrm{g}}V+V\text{Ric}_{\mathrm{g}}= 0.
\end{equation}
The unique static vacuum black hole is the Schwarzschild solution, whose time slice is the static manifold $(\mathbb{M}^3_m=\mathbb{R}^3\setminus \left\{|x|<\frac{m}{2}\right\},\mathrm{g}_m,V_m)$ with
\begin{equation}
\mathrm{g}_m=\left(1+\frac{m}{2|x|}\right)^{4}\delta,\qquad \quad V_m=\frac{1-\frac{m}{2|x|}}{1+\frac{m}{2|x|}},
\end{equation}
where $m$ is the ADM mass parameter.

In \cite{ChenWangYau18} Chen, Wang, Wang, and Yau defined a version of the Wang-Yau quasi-local energy with respect to a general static spacetime. In particular, let $(\Sigma,\sigma,|\vec{H}|,\alpha)$ be a Wang-Yau data set and assume that there is an isometric embedding $\iota:\Sigma\hookrightarrow\mathbb{M}^3$ so that the time function $\tau=0$, then the static Liu-Yau mass is given by
\begin{equation}
m_{LY}^S(\Sigma)=\frac{1}{8\pi}\int_{\Sigma}V\left(H_s-|\vec{H}|\right)dA_{\sigma},
\end{equation}
where $H_s$ is the mean curvature of the isometric embedding $\iota(\Sigma)$. If the data bound a spacelike hypersurface $(\Omega,g,k)$ with vanishing extrinsic curvature $k=0$, then the static Brown-York mass takes the form
\begin{equation}
m_{BY}^S(\Sigma)=\frac{1}{8\pi}\int_{\Sigma}V\left(H_s-H\right)dA_{\sigma},
\end{equation}
where $H$ is the mean curvature of the embedding $\Sigma\hookrightarrow\Omega$.
Recently, Lu and Miao \cite{LuMiao} have proven a Penrose type inequality
for the static Brown-York mass in which the reference static spacetime is the Schwarzschild solution. Here we give an extension of their result to the static Liu-Yau mass setting.

\begin{theorem}\label{thm7.1}
Let $(\Omega,g,k)$ be an initial data set for the Einstein equations satisfying the dominant energy condition $\mu\geq |J|_g$ which is strict on horizons. Suppose that $\Omega$ is compact with boundary consisting of a disjoint union $\partial\Omega=\Sigma_h \cup\Sigma$ where $\Sigma_h$ is a (nonempty) apparent horizon, no other closed apparent horizons are present, and the single component $\Sigma$ is untrapped. If $\Sigma$ isometrically embeds into a Schwarzschild manifold $(\mathbb{M}^3_m,\mathrm{g}_m,V_m)$ as a star-shaped 2-convex surface on which $\text{Ric}_{\mathrm{g}_m}(\nu,\nu)\leq 0$,
then there exists a nonzero positive constant $\gamma$ independent of horizon area such that
\begin{equation}\label{LYstatic1in}
m+\frac{1}{8\pi}\int_{\Sigma}V_m(H_m -|\vec{H}|) dA_{\sigma}\geq \frac{\gamma}{1+\gamma}\sqrt{\frac{|\Sigma_{h}|}{4\pi}},
\end{equation}
where $H_m$ is the mean curvature of the embedding into $\mathbb{M}^3_m$.
Moreover, if in addition the enhanced energy condition $\mu_{EM}\geq|J_{EM}|_g$ holds, $(\Omega,g,k,E,B)$ and $\partial\Omega$ are axially symmetric, and $\Sigma_h$ is stable in the sense of apparent horizons then
\begin{equation}\label{LYstatic2in}
m+\frac{1}{8\pi}\int_{\Sigma}V_m(H_m -|\vec{H}|)dA_{\sigma}\geq \frac{\gamma}{1+\gamma}\left(\sum_{i=1}^I \sqrt{(Q_{h}^i)^4+4(\mathcal{J}_{h}^i)^2}\right)^{\frac{1}{2}}.
\end{equation}
\end{theorem}

\begin{proof}
This result follows by combining the proof of Theorem \ref{thm2.7} above, and the proof of Theorem 1.1 in \cite{LuMiao}. Here we provide an outline. Let $(\Omega,\bar{g}, \bar{k},X,\bar{E})$ be the Jang deformation of the initial data with the solution of Jang's equation blowing-up at the horizon and satisfying the Dirichlet boundary condition $f=0$. A Bartnik-Shi-Tam extension $(\bar{M}_+,\bar{g}_+,\bar{k}_+ =0,
X_+ =0, \bar{E}_+ =0)$ of the Jang initial data may be constructed in the following way with zero scalar curvature and boundary $\partial\bar{M}_+ =\Sigma$ having mean curvature
\begin{equation}\label{879065}
H_+ = \bar{H}-X(\bar{\nu}).
\end{equation}
The 2-convex condition means that the two elementary symmetric polynomials of the principal curvatures $\kappa_1$, $\kappa_2$ of the embedding $\Sigma\hookrightarrow\mathbb{M}^3_m$ are positive, and this together with $\text{Ric}_{\mathrm{g}_m}(\nu,\nu)\leq 0$ implies via the Gauss equations that $\Sigma$ has positive Gauss curvature and is thus topologically a sphere. These hypotheses together with the star-shaped assumption lead to smooth long time existence of the flow inside $\mathbb{M}^3_m$ defined by
\begin{equation}
\partial_r =\kappa\nu,\quad\quad\quad\quad \kappa=\frac{1}{4}\left(\kappa_1^{-1}+\kappa_2^{-1}\right),
\end{equation}
with leaves denoted by $\{\Sigma_r\}_{r=0}^{\infty}$ and $\Sigma_0=\Sigma$.
The region of the Schwarzschild manifold foliated by this flow is $(\bar{M}_+, \kappa^2 dr^2 +\sigma_r)$, where $\sigma_r$ is the induced metric on the leaves $\Sigma_r$. The extension metric then takes the form $\bar{g}_+=\mathrm{u}^2 dr^2 +\sigma_r$ for an appropriately chosen function $\mathrm{u}:\bar{M}_+ \rightarrow\mathbb{R}_+$, with
$\mathrm{u}_0=\left(\bar{H}-X(\bar{\nu})\right)^{-1} \kappa H_m$ in order to ensure \eqref{879065}. As shown in \cite{LuMiao} the function
\begin{equation}
m+\frac{1}{8\pi}\int_{\Sigma_r}V_m(H_{m,r} - H_{\mathrm{u},r})dA_{\sigma_r}
\end{equation}
is monotonically nonincreasing in $r$, where $H_{m,r}$ and $H_{\mathrm{u},r}$ are the mean curvatures of $\Sigma_r$ with respect to the Schwarzschild metric and $\bar{g}_+$ respectively. Furthermore this function converges to the mass $\bar{\mathbf{m}}$ of the extension metric $\bar{g}_+$ as $r\rightarrow\infty$. It then follows from \eqref{tytyty} that
\begin{equation}
m+m_{LY}^S(\Sigma)\geq m+\frac{1}{8\pi}\int_{\Sigma_r}V_m \left[H_{m} - \left(\bar{H}-X(\bar{\nu})\right)\right]dA_{\sigma}
\geq \bar{\mathbf{m}}
\geq \frac{\gamma}{1+\gamma}
\sqrt{\frac{|\Sigma_{h}|}{4\pi}}.
\end{equation}
This establishes \eqref{LYstatic1in}. As in the proof of Theorem \ref{thm2.7}, inequality \eqref{LYstatic2in} is then obtained by applying the area-angular momentum-charge inequality \cite{BrydenKhuri,ClementJaramilloReiris}.
\end{proof}

In a similar manner, by combining the proofs of Theorem \ref{thm2.9} above and Theorem 1.1 in \cite{LuMiao} we obtain the following result.

\begin{theorem}\label{thm7.2}
Let $(\Omega,g,k,E,B)$ be an axisymmetric initial data set for the Einstein-Maxwell equations satisfying \eqref{frobenius}, $J(\eta)=0$, and the energy condition $\mu_{EM}\geq |J|_g$ which is strict on horizons. Suppose that $\Omega$ is compact with axisymmetric boundary consisting of a disjoint union $\partial\Omega=\Sigma_h \cup\Sigma$ where $\Sigma_h$ is a single component (nonempty) apparent horizon, no other closed apparent horizons are present, and the single component $\Sigma$ is untrapped. If $\Sigma$ isometrically embeds into a Schwarzschild manifold $(\mathbb{M}^3_m,\mathrm{g}_m,V_m)$ as a star-shaped 2-convex surface on which $\text{Ric}_{\mathrm{g}_m}(\nu,\nu)\leq 0$,
then there exists a nonzero positive constant $\lambda$ independent of horizon area such that
\begin{equation}\label{LYstatic3in}
\left(m+m_{LY}^S(\Sigma)\right)^2\geq \left(\frac{\gamma}{1+\gamma}\sqrt{\frac{|\Sigma_{h}|}{4\pi}}
+\lambda \sqrt{\frac{\pi}{|\Sigma_h|}}Q^2\right)^2
+\frac{\lambda\gamma}{1+\gamma}\frac{8\pi^2\mathcal{J}^2}{\mathcal{C}^2},
\end{equation}
where $\gamma$ is as in Theorem \ref{thm7.1}. Moreover, the same inequality holds without the assumption of axisymmetry if the contribution from angular momentum is removed.
\end{theorem}

\begin{remark}
As in Section \ref{sec7} this proof may be parlayed into versions of Theorems \ref{thm2.11} and \ref{thm2.12} for the static Wang-Yau mass. In addition,
we note that the same methods can be used to generalize the result of P.-N. Chen \cite{Chen}, concerning Brown-York mass with general spherically symmetric static reference, to the setting of static Liu-Yau mass.
\end{remark}

\end{document}